\documentclass[11pt]{article}

\usepackage[margin=1.27in]{geometry}

\usepackage[misc]{ifsym} 
\usepackage{lipsum}
\usepackage{amssymb}
\usepackage{amsmath}
\usepackage{latexsym}
\usepackage{bm}
\usepackage{enumitem}
\usepackage{graphicx}
\usepackage{amscd}
\usepackage{extarrows}
\usepackage{amsfonts}
\usepackage{indentfirst}
\usepackage{bbm}
\usepackage{mathdots}
\usepackage{mathtools}
\usepackage[linktocpage]{hyperref}
\usepackage{amsthm}
\usepackage{ upgreek }
\hypersetup{
    colorlinks,
    citecolor=black,
    filecolor=black,
    linkcolor=blue,
    urlcolor=black
}

\usepackage{quiver}
\usepackage[multiple]{footmisc}
\usepackage[title]{appendix}
\usepackage{setspace}

\newcommand\blfootnote[1]{%
  \begingroup
  \renewcommand\thefootnote{}\footnote{#1}%
  \addtocounter{footnote}{-1}%
  \endgroup
}

\usepackage{float}

\theoremstyle{plain}
\newtheorem{theorem}{Theorem}[section]
\newtheorem{proposition}[theorem]{Proposition}
\newtheorem{lemma}[theorem]{Lemma}
\newtheorem{corollary}[theorem]{Corollary}
\newtheorem*{theorem*}{Theorem}
\newtheorem*{lemma*}{Lemma}

\theoremstyle{remark}
\newtheorem{remark}[theorem]{Remark}

\theoremstyle{definition}
\newtheorem{definition}{Definition}[section]
\newtheorem{example}[theorem]{Example}

\DeclareMathOperator{\im}{Im}
\DeclareMathOperator{\Hom}{Hom}

\DeclareMathOperator{\tr}{tr}

\DeclareMathOperator{\Tr}{Tr}
\DeclareMathOperator{\Nm}{Nm}
\DeclareMathOperator{\id}{id}
\DeclareMathOperator{\Ind}{Ind}

\DeclareMathOperator{\Gal}{Gal}

\DeclareMathOperator{\sgn}{sgn}
\DeclareMathOperator{\ch}{ch}

\DeclareMathOperator{\SL}{SL}
\DeclareMathOperator{\GL}{GL}
\DeclareMathOperator{\PGL}{PGL}
\DeclareMathOperator{\PSL}{PSL}
\DeclareMathOperator{\nrd}{nrd}
\DeclareMathOperator{\St}{St}
\DeclareMathOperator{\SO}{SO}

\DeclareMathOperator{\JL}{JL}
\DeclareMathOperator{\SU}{SU}

\DeclareMathOperator{\Mt}{Mat}
\DeclareMathOperator{\PSO}{PSO}

\DeclareMathOperator{\rdch}{rdch}
\DeclareMathOperator{\tp}{temp}
\DeclareMathOperator{\covol}{covol}

\DeclareMathOperator{\disc}{disc}
\DeclareMathOperator{\ram}{Ram}
\DeclareMathOperator{\N}{N}
\DeclareMathOperator{\nr}{nr}

\DeclareMathOperator{\Stb}{Stab}
\DeclareMathOperator{\Sp}{Sp}
\DeclareMathOperator{\rk}{rank}

\title{An Arithmetic Invariant of the Jacquet-Langlands correspondence}
\author{Jun Yang}
\date{}

\begin{document}
\maketitle
\begin{abstract} 
We describe the local-global compatibility of local Plancherel measures and the Tamagawa measure under the Jacquet-Langlands correspondence. 
We apply the notion of densities of modules over discrete groups, which generalizes the dimensions over discrete groups. 
We prove that the global Jacquet-Langlands correspondence preserves densities over principal arithmetic groups. 
\end{abstract}

\blfootnote{\Letter~ \href{mailto:junyang@fas.harvard.edu}{junyang@fas.harvard.edu}~~~~~~~Harvard University, Cambridge, MA 02138, USA}
\blfootnote{This work was supported in part by the ARO Grant W911NF-19-1-0302 and the ARO MURI Grant W911NF-20-1-0082.}

\tableofcontents

\section{Introduction}

In the 1970s, Jacquet and Langlands proposed a correspondence between the square-integrable irreducible representations of $\GL(2,F_v)$ over a local field $F_v$, and the finite-dimensional irreducible representations of the multiplicative group $D^{\times}(F_v)$ of a non-split 
quaternion $F_v$-algebra $D(F_v)$. 
This is the classical {\it local Jacquet-Langlands correspondence} for $\GL(2)$. 
If $\pi_v,\pi'_v$ correspond and $F_v$ is non-archimedean, they proved that the formal degree of $\pi_v$ and the usual complex dimension of $\pi'_v$ are always proportional. 
More precisely, if $\JL_v$ denotes the local Jacquet-Langlands correspondence from the unitary dual of $D^{\times}(F_v)$ to the one of $\GL(2,F_v)$, we have
\begin{equation}\label{eintrodeg}
\begin{aligned}
     \frac{\deg(\pi_v)}{\dim \pi'_v}=c_v, \text{~if~} \JL_{v}(\pi')=\pi,
\end{aligned}
\end{equation}
where  $c_v$ is a constant that does not depend on  $\pi_v$ or $\pi'_v$, but depends on the Haar measures on $\GL(2,F_v)$ and $D^{\times}(F_v)$. 

The local correspondence has been extended to $\GL(n)$ by Rogawski \cite{Rog80} and Deligne-Kazhdan-Vigneras \cite{DKV84}.  
Aubert and Plymen extend the property \eqref{eintrodeg} to the most general case that the Plancherel measures of corresponding tempered representations are also proportional, i.e.,
\begin{equation}\label{eintroP}
\begin{aligned}
    \frac{d \nu_{v}(\pi_v)}{d \nu'_{v}(\pi'_v)}=c_v, \text{~if~} \JL_{v}(\pi')=\pi,
\end{aligned}
\end{equation}
where $\nu_v$ (resp. $\nu'_v$) denotes the Plancherel measure of $\GL(n,F_v)$ (resp. an inner form $G'(F_v)$ of $\GL(n,F_v)$) and $c_v$ also depends only on the Haar measures.

On the other hand, the {\it global Jacquet-Langlands correspondence} matched certain automorphic representations of adelic $\GL(n)$ and those of its inner form, both over a global field $F$.
It is compatible with each local correspondence. 
It was also first proposed for $\GL(2)$ by Jacquet and Langlands and then generalized to $\GL(n)$ by Badulescu \cite{Badu08} (see also \cite{BaduRena10,BaduRoch2017}). 

One would expect a similar property as Equation \eqref{eintrodeg} or \eqref{eintroP} for the global correspondence. 
There are mainly two difficulties. 
\begin{enumerate}
    \item For automorphic representations of adelic groups, neither the formal degrees nor the global Plancherel measure is well-defined because of the divergence of the product of the local formal degrees and local Plancherel measures over all places. 
    \item Although we can freely renormalize local Haar measures, there is a unique canonical Haar measure, the Tamagawa measure, on the product of local groups--the adelic group. 
    This measure is defined naturally and studied in the proofs of the Weil conjecture on Tamagawa numbers (see \cite{Kott88}). 
\end{enumerate}
We may ask if there is a family of local Haar measures $\mu_v,\mu'_v$ on local groups such that $c_v=1$ in Equation \eqref{eintroP} for all $v$, which is compatible with the products $\prod\mu_v,\prod\mu'_v$ being Tamagawa measures?

To be more precise, we let $D$ be a central division algebra over a global field $F$ of index $d$ 
\footnote{The central division $F$-algebra $D$ is called of index $d$ if $\dim_F D=d^2$}. 
Consider the $F$-groups $G=\GL(nd,F)$ and its inner form $G'=\GL(n,D)$. 
For each place $v$, we know $\Mt(n,D)\otimes_{F}F_v\cong \Mt(n_v,D_v)$ for some $F_v$-division algebra $D_v$ of index $nd/n_v$.  
Let $G_v=\GL(nd,F_v)$ and $G'_v=\GL(n_v,D_v)$. 
For any given Haar measures $\mu_v,\mu'_v$ on $G_v,G'_v$, we let $\nu_v,\nu'_v$ be the associated Plancherel measure on the tempered duals. 
The following compatibility result is first proved (see Theorem \ref{tpreallPm}).
\begin{lemma}\label{lmain1}
There exists local Haar measures $\mu_v,\mu'_v$ on $G_v,G'_v$ respectively such that 
\begin{enumerate}
    \item $\prod \mu_v,\prod \mu'_v$ are the Tamagawa measures on $\GL(nd,\mathbb{A})$ and $\GL(n,D(\mathbb{A}))$ respectively;
    \item the local Jacquet-Langlands correspondence preserves the associated Plancherel measure, i.e., for each $v$, 
    \begin{center}
        $d \nu_{v}(\pi_v)=d \nu'_{v}(\pi'_v)$ 
    \end{center} 
    if $\pi_v$ and $\pi'_v$ correspond. 
\end{enumerate}
\end{lemma}

Based on this result, we consider square-integrable representations as well as tempered representations as modules over the arithmetic subgroups and their group von Neumann algebras. 
For a discrete group $\Gamma$, 
we adapt the notion $d_{\Gamma}(\pi)$ for the $\Gamma$-density of a representation $(\pi,H)$ over $\Gamma$, which extends the definition of $\Gamma$-dimensions (see Section \ref{ssvndim}). 
The twisted version $d_{\Gamma,\sigma}(\pi)$ of $\Gamma$-density is also defined for $\sigma$-projective representations associated to a $2$-cocycle $\sigma$ of $\Gamma$ (see Defitinion \ref{dvnden}).

Let $S$ be a finite set of places of $F$ containing all archimedean ones. 
Define $G_S=\prod_{v\in S}G_v,G'_S=\prod_{v\in S}G'_v$ and
let $\Gamma_S=\PGL(nd,\mathcal{O}_S)$ (resp. $\Gamma'_S=\PGL(n,D(\mathcal{O}_S))$), which is the principal $S$-arithmetic subgroup of the semisimple group $G/Z(G)=\PGL(nd)$ (resp. $G'/Z(G')=\PGL(n,D)$).
Let $\pi=\otimes \pi_v,\pi'=\otimes \pi'_v$ be cuspidal automorphic representations of $G(nd,\mathbb{A})$ and $\GL(n,D(\mathbb{A}))$.
We let $\pi_S,\pi'_S$ be the $S$-components of $\pi,\pi'$, i.e., $\pi_S=\otimes_{v\in S}\pi_v$ and  $\pi'_S=\otimes_{v\in S}\pi'_v$. 
We let $\sigma(\pi_S)$ and $\sigma(\pi'_S)$ be the $2$-cocycles of $\Gamma_S,\Gamma'_S$ obtained by passing the $\pi_S,\pi'_S$ to the projective representations (see Section \ref{B2}). 
Here is the main result (see \ref{ttpvndim}). 
\begin{theorem}\label{tmain2}
Suppose $\pi,\pi'$ are tempered cuspidal automorphic representations of $\GL(nd,\mathbb{A})$ and $\GL(n,D(\mathbb{A}))$ that are matched by the global Jacquet-Langlands correspondence.   
We have
\begin{center}
    $d_{\Gamma_S,\sigma(\pi_S)}(\pi_S)=d_{\Gamma'_S,\sigma(\pi'_S)}(\pi'_S)$, 
\end{center}
if $S$ contains the ramification set of $G'$ and both $G_S,G'_S$ are non-compact modulo their centers. 
\end{theorem}

Note that the Ramanujan-Petersson conjecture predicts that all the cuspidal automorphic representations are tempered.

Let us outline the proof of these results as well as the article. 
We review some basic facts in Section \ref{spre} for the division algebras, the Jacquet-Langlands correspondences, von Neumann dimensions, and $\Gamma$-densities.  
We then discuss the local correspondences for the archimedean and non-archimedean places in Section \ref{sJLmeas}. 
Then Lemma \ref{lmain1} is proved at the end of Section \ref{sJLmeas} and Theorem \ref{tmain2} is proved in Section \ref{sJLinv}.

\section{Preliminaries}\label{spre}

We review some basic facts about division algebras, the Jacquet-Langlands correspondences of $\GL(n)$, the Plancherel formula for central extensions, von Neumann dimensions, and introduce the density over discrete groups. 
Some notations are introduced and will be used in the next sections. 


\subsection{Division algebras over global fields}\label{ssdivnumfield}
 
Let $F$ be a global field and $\mathcal{O}_F$ be the ring of integers of $F$. 
Let $V$ be the set of places of $F$, $V_{\infty}$ be the set of infinite places of $F$ (which is empty if $F$ is a function field), and $V_{f}$ be the set of finite places of $F$. 
For a place $v$ of $F$, we let $F_v$ be the completion of $F$ with respect to $v$.
For $v\in V_f$, let $P_v$ be the prime ideal in $\mathcal{O}_F$ associated with $v$ and $\mathcal{O}_v$ be the ring of integers of $F_v$.  
Let $\mathfrak{m}_v$ be its maximal ideal and $k(v)=\mathcal{O}_v/\mathfrak{m}_v\cong\mathcal{O}_F/P_v$ be its residue field. 
Let $q_v$ be the cardinality of $k(v)$.

Let $A$ be a finite-dimensional simple algebra over $F$. 
By the Wedderburn-Artin theorem (see \cite[Theorem 7.4]{Rei03}), we have $A\cong \Mt(n,D)$ for some division algebra $D$ over $F$ such that the center $Z(D)\cong F$.  
We know that $\dim_F D=d^2$ for some positive integer $d$.

For each $v\in V$, let $A_v=A\otimes_F F_v$. 
There exists a division algebra $D_v$ 
\footnote{Note $D_v$ is not the algebra $D\otimes_{F}F_v$ in general unless $n=1$. } 
over $F_v$ such that
\begin{center}
    $A_v\cong \Mt(n_v,D_v)$
\end{center}
for some $n_v\geq 1$, which is a multiple of $n$ such that $nd=n_v d_v$ ($\dim_{F_v}D_v=d_v^2$). 
We say $A$ (or $A^{\times}$) {\it splits at $v$} if $d_v=1$ (or, equivalently, $n_v=nd$), in which case $A_v\cong \GL(nd,F_v)$. 
We let $G'=A^{\times}=\GL(n,D)$ and
\begin{itemize}
    \item $\ram(G')$ = the set of places at which $A$ (or $G'$) does not split.
\end{itemize}
It is known that $\ram(G')$ is a finite subset of $V$. 
For $v\in V_f$, the following notations will be used throughout this paper (see \cite{Voight21}).  
Recall that an order is an $\mathcal{O}_v$-lattice as well as a sub-ring. 
\begin{itemize}
    \item $O(D_v),O(A_v)$ = the maximal order of $D_v$ and $A_v$, respectively. 
    \item $J(D_v),J(A_v)$ = the maximal ideal of $O(D_v)$ and $O(A_v)$, respectively. 
\end{itemize}

In the rest of this section, we assume $K=F$ or $F_v$ for some $v\in V$ and denote the $K$-algebra $A$ or $A_v$ simply by $A$. 
For $a\in A$, let $\lambda_a\in \Hom_{K}(A,A)$ be the left multiplication by $a$. 
We denote by $\ch_{A/K}(a;X)$ be the characteristic polynomial of $\lambda_a$ and  write $\ch_{A/K}(a;X)=\sum_{0\leq i\leq m}c_i(a) X^i$ with each $c_i(a)\in K$ and $c_m(a)=1$. 
We follow \cite[\S 2]{Rei03} for these definitions and notations.  
\begin{enumerate}
    \item The \textbf{trace} $\Tr_{A/K}\colon A\to K$ is defined to be $\Tr_{A/K}(a)=-c_{m-1}(a)$;
    \item The \textbf{norm} $\Nm_{A/K}\colon A\to K$ is defined to be $\Nm_{A/K}(a)=(-1)^m c_{0}(a)$. 
\end{enumerate}
It is known that there exists a finite extension $K_1$ of $K$ such that $A\otimes_K K_1\cong \Mt(m,K_1)$. 
The reduced characteristic polynomial $\rdch_{A/K}(a;X)$ is defined to be the characteristic polynomial of $a\otimes 1$ under this isomorphism. 
It is known that $\rdch(a;X)\in K[X]$ for each $a\in A$ and $\rdch_{A/K}(a;X)^m=\ch_{A/K}(a;X)$. 
\begin{enumerate}
\setcounter{enumi}{2}
    \item The \textbf{reduced trace} $\tr\colon A\to K$ is defined by $\tr(a)=\Tr(a\otimes 1)=\frac{\Tr(a)}{m}$.
    \item The \textbf{reduced norm} $\nr\colon A\to K$ is defined by $\nr(a)=\Nm(a\otimes 1)=\Nm(a)^{1/m}$.
\end{enumerate}


Let $R$ be a noetherian domain and $K$ be the fraction field of $R$.
Let $B$ be a semisimple algebra over $K$ and $I
\subset B$ be an $R$-lattice. 
Suppose $\dim_K B=m$. 
For $\alpha_1,\cdots,\alpha_m\in B$, we let $d(\alpha_1,\cdots,\alpha_m)=\det([\tr(\alpha_i\alpha_j)]_{i,j})$.  

\begin{enumerate}
\setcounter{enumi}{4}
\item the \textbf{discriminant} $\disc_{R}(I)$ is the $R$-submodule of $K$ generated by the set
\begin{center}
    $\{d(\alpha_1,\cdots,\alpha_m)\}$. 
\end{center}
It can be shown that if $I$ has a free $R$-basis $e_1,\cdots,e_m$, then $\disc_R(I)=\det([\tr(\alpha_i\alpha_j)]_{i,j})\cdot R$, which is a principal ideal (see \cite[\S 10]{Rei03}). 

\item the \textbf{absolute norm} $\N(\mathfrak{a})$ of a fractional ideal $\mathfrak{a}$ of $R$ is defined to be $\N(\mathfrak{a})=[R\colon \mathfrak{a}]$.


\item the \textbf{norm} on $A_v$ is defined to be $\|x\|_{A_v}=\|\Nm_{A_v/F_v}(x)\|_{F_v}$. 
\end{enumerate}

\subsection{The Jacquet-Langlands correspondences of $\GL(n)$}

Let $D$ be a central division algebra over $F$ of $F$-dimension $d^2$. 
Let $G=\GL(nd,F)$, $G'=A^{\times}=\GL(n,D)$ and $\mathbb{A}$ be the adele ring of $F$. The following notations will be used throughout this paper.
\begin{itemize}
    \item $G_v=\GL(nd,F_v)$ and $G_{\mathbb{A}}=\GL(nd,\mathbb{A})$
    \item $G'_v=A_v^{\times}=\GL(n_v,D_v)$ and $G'_{\mathbb{A}}=\GL(n,D(\mathbb{A}))$. Note that $\dim_{F_v}D_v=d_v^2$ and $n_v d_v =nd$,  
    \item $Z(\mathbb{A})=$ the center of $G_{\mathbb{A}}$, which is also the center of $G'_{\mathbb{A}}$.  
    \item $\Pi_{0}(G),\Pi_{0}(G')$ = the set of equivalence classes of irreducible cuspidal subrepresentations of $L^2(G(F)Z(\mathbb{A})\backslash
    G(\mathbb{A}))$ and $L^2(G'(F)Z(\mathbb{A})\backslash
    G'(\mathbb{A}))$ respectively. 
    \item $\Pi(H),\Pi_{\tp}(H),\Pi_{2}(H)$ \footnote{If the center of a group $H$ is not compact, there will be no square-integrable representation and we usually take the irreducible representations which are square-integrable irreducible modulo the center of $H$} = the set of equivalence classes of unitary, tempered and square-integrable irreducible representations of a group $H$ respectively. 
    \item $\mathcal{H}(G_v)$ = the Hecke algebra of the group $G_v$ (the algebra of smooth or locally constant compactly supported complex functions on $G_v$ according to $F_v$ is archimedean or not, with the multiplication given by the convolution). 
    $\mathcal{H}(G'_v)$ is defined similarly. 
\end{itemize}

For a linear algebraic group $H$ over a local field $K$, 
we let $H^s$ be the set of regular semisimple \footnote{An element $g\in H$ is called {\it regular semisimple} if the characteristic polynomial of $g$ has distinct roots over an algebraic closure of $K$} elements of $H$. 
For a unitary irreducible representation $\pi$ of $H$, 
we let $\theta_{\pi}$ be the character of $\pi$ defined on the set $H^s$ which is given by 
\begin{center}
   $\theta_{\pi}\stackrel{\rm df}{=}\tr\pi(f)=\int_{H^s}\theta_{\pi}(g)f(g)dg$,  
\end{center}
where $f$ belongs to the Hecke algebra of $H$.  
For $g\in G_v$ and $g'\in G'_v$, we say $g$ corresponds to $g'$, denoted by $g\leftrightarrow g'$, if they have the same characteristic polynomial over $F_v$. 
\begin{theorem}
\begin{enumerate}
    \item For each $v\in V$, there exists a unique injective map
    \begin{center}
        $\JL_v\colon \Pi(G'_v)\to \Pi(G_v)$
    \end{center}
    such that for each $\pi'_v\in \Pi(G'_v)$ we have $\theta_{\pi'_v}(g)=(-1)^{nd-n}\theta_{\JL_{v}(\pi'_v)}(g')$ for all $g\in G_v$ and $g'\in G'_v$ such that $g'\leftrightarrow g$. 
    \item There exists a unique map 
    \begin{center}
        $\JL\colon \Pi_{0}(G') \to \Pi_{0}(G)$
    \end{center}
    such that if $\pi=\otimes'_{v\in V} \pi_v=\JL(\pi')$ for some $\pi'=\otimes'_{v\in V}\pi'_v$,  
    then
    \begin{enumerate}
        \item $\JL_v(\pi_v)=\pi'_v$ if $v\in \ram(G')$;
        \item $\pi_v\cong \pi'_v$ if $v\notin \ram(G')$. 
    \end{enumerate}
\end{enumerate}
\end{theorem}

This theorem is proved for $\GL(2)$ by Jacquet and Langlands in \cite{JL70} and extended to $\GL(n) $ by Badulescu and Renard in \cite{Badu08,BaduRena10} for number fields.
The function field case is proved by Badulescu and Roche in \cite{BaduRoch2017}.

\subsection{The Plancherel formula for central extensions}\label{B2}

Let $G$ be a unimodular locally compact group of type I. 
Following \cite[\S 7.2]{Fo2}, a group is called type I if each primary representation\footnote{A unitary representation $(\pi,H)$ of $G$ is called primary if $\pi(G)''$, the von Neumann algebra it generates, is a factor, i.e., $Z(\pi(G)'')\cong \mathbb{C}$. 
Equivalently, assuming $(\pi,H)$ is a direct sum of irreducible representations, $(\pi,H)$ is primary if and only if $(\pi,H)$ is a direct sum of a single irreducible representation}  
 generates a type I factor. 
More precisely, given any unitary representation $(\pi,H)$ of $G$, if $\pi(G)''$ is factor, then $\pi(G)''$ is a factor of type I, i.e. $\pi(G)''\cong B(K)$ for some Hilbert space $K$ (possibly infinite-dimensional).
The class of type I groups contains real linear algebraic groups (see \cite[\S 8.4]{Kiri76}), 
reductive $p$-adic groups (see \cite{Bnst74}), and
also reductive adelic groups (see \cite[Appendix]{Clz07}).

For a 2-cocycle $\sigma\colon G\times G\to \mathbb{T}$ of $G$, we let
$\Pi(G,\sigma)$ be the set of equivalence classes of $\sigma$-projective irreducible representations of $G$:
\begin{center}
    $\pi(g)\pi(h)=\sigma(g,h)\pi(gh)$ for each $\pi \in \Pi(G,\sigma)$. 
\end{center} 
We also let $\lambda_{\sigma}$ and $\rho_{\sigma}$ be the $\sigma$-projective left and right regular representation of $G$ on $L^2(G)$ as follows: 
for $g\in G$, $f\in L^2(G)$, 
\begin{center}
    $\lambda_{\sigma}(g)f(x)=\sigma(x^{-1},g)f(g^{-1}x)$ and $\rho_{\sigma}(g)f(x)=\sigma(g^{-1},x^{-1})f(xg)$. 
\end{center}
We let $\overline{\sigma}$ denote the complex conjugate of $\sigma$. 

The following result is an altered statement of the theorem proved by Kleppner and Lipsman (see \cite[I.Theorem 7.1]{KlepLips1972}). 
Note that each group we consider here is unimodular and of type I. 

\begin{theorem}\label{tprojPthm}
There exists a positive standard Borel measure $\nu_{G,\sigma}$ on $\Pi(G,\sigma)$, a measurable field $\zeta\mapsto H_{\zeta}\otimes H_{\zeta}^{*}$ and a measurable field of representations $\zeta\mapsto \pi_{\zeta}$ such that
\begin{enumerate}
    \item $\pi_{\zeta}\in \zeta$ for $\nu_{G,\sigma}$-almost all $\zeta$;
    \item an isomorphism $\Psi\colon L^2(G)\to \int_{\Pi(G,\sigma)}^{\oplus}H_{\zeta}\otimes H_{\zeta}^{*}d \nu_{G,\sigma}(\zeta)$ which intertwines
    \begin{enumerate}
        \item the $\sigma$-projective representations $\lambda_{\sigma}$ with $\int_{\Pi(G,\sigma)}^{\oplus}\pi_{\zeta}\otimes \id_{\zeta}d \nu_{G,\sigma}(\zeta)$;
        \item the $\overline{\sigma}$-projective representations $\rho_{\sigma}$ with $\int_{\Pi(G,\sigma)}^{\oplus}\id_{\zeta}\otimes\pi_{\zeta}^{*} d \nu_{G,\sigma}(\zeta)$;
    \end{enumerate}
    \item $(\Psi f)(\zeta)=\pi_{\zeta}(f)$ for $f\in L^1(G)$;
    \item $\int_{G}f(x)\overline{h}(x)d \mu_{G}(x)=\int_{\Pi(G,\sigma)}\Tr(\pi_{\zeta}(f)\pi_{\zeta}(h)^{*})d \nu_{G,\sigma}(\zeta)$ for $f,h\in L^1(G)$. 
\end{enumerate}
\end{theorem}

Let $Z(G)$ be the center of $G$ and $\overline{G}=G/Z(G)$. 
Let $\chi$ be a character of $Z(G)$, i.e., $\chi\in \Hom(Z(G),\mathbb{T})$.   
There exists a 2-cocycle $\sigma_{\chi}$ of $\overline{G}$ with the following properties: if we denote again by $\sigma_{\chi}$ the lift to $G$, then $\chi$ extends to an $\sigma_{\chi}$-projective representation $\chi'$ of $G$. 
It is known that $\sigma_{\chi}$ is unique up to $2$-coboundaries and $\sigma_{\chi}$ can be always be normalized in the sense that $\sigma_{\chi}(g,g^{-1})=1$ (see \cite[I.\S 1]{KlepLips1972}). 
We fix the $2$-cocycle $\sigma_{\chi}$ and let $\tau$ be a $\sigma_{\chi}$-projective representation of $\overline{G}$. 
Denote by $\tilde{\tau}$ its lift to $G$. 
We define
\begin{center}
    $\pi(\tau,\chi)=\sigma_{\chi}\otimes \tilde{\tau}$
\end{center}
which is an ordinary representation of $G$.  

The following result gives the Plancherel measure of central extension, which is a special case of \cite[Theorem 3]{KlpLips1973} (see also \cite[I Theorem 8.1, Theorem 10.2]{KlepLips1972}). 
Here $\mu_G$ denotes a Haar measure on $G$. 
\begin{theorem}\label{tcenextPfml}
The left regular representation $\lambda_G$ of $G$ on $L^2(G,\mu_G)$ can be decomposed as the following direct integral:
\begin{center}
$L^2(G,\mu_G)=\int_{\Pi(Z(G))}^{\oplus}\int_{\Pi(\overline{G},\overline{\sigma_{\chi}})}^{\oplus}\pi(\tau,\chi) \otimes \id_{\pi(\tau,\chi)}d\nu_{\overline{G},\overline{\sigma_{\chi}}}(\tau)d \nu_{Z(G)}(\overline{\chi})$.
\end{center}
The right regular representation $\rho_G$ can be decomposed in a similar way. 
Moreover, for $f\in L^1(G)\cap L^2(G)$, we have
\begin{center}
    $\int_{G}|f(g)|^2d \mu_G(g)=\int_{\Pi(Z(G))}\int_{\Pi(\overline{G},\overline{\sigma_{\chi}})}\|\pi(\tau,\chi)(f)\|_2^{2}d\nu_{\overline{G},\overline{\sigma_{\chi}}}(\tau)d \nu_{Z(G)}(\overline{\chi})$. 
\end{center}
\end{theorem}

\subsection{The density over discrete groups}\label{ssvndim}

Let $\Gamma$ be a countable group with the counting measure.  
Let $\{\delta_{\gamma}\}_{\gamma\in \Gamma}$ be the usual orthonormal basis of $l^2(\Gamma)$. 
We also let $\lambda$ and $\rho$ be the left and right regular representations of $\Gamma$ on $l^2(\Gamma)$, respectively.
For all $\gamma,\gamma'\in \Gamma$, we have
$\lambda(\gamma')\delta_{\gamma}=\delta_{\gamma'\gamma}$ and $\rho(\gamma')\delta_{\gamma}=\delta_{\gamma\gamma'^{-1}}$. 
Let $\mathcal{L}(\Gamma)$ be the strong operator closure of the complex linear span of $\lambda(\gamma)$'s (or equivalently,  $\rho(\gamma)$'s). 
This is the {\it left group von Neumann algebra of $\Gamma$}. 
The {\it right group von Neumann algebra of $\Gamma$}, denoted by $\mathcal{R}(\Gamma)$, is defined in a similar way by the right regular representation.

We fix a $2$-cocycle $\sigma\colon \Gamma\times \Gamma \to \mathbb{T}$ on $\Gamma$, which is normalized in the sense that $\sigma(\alpha,\beta)\sigma(\alpha\beta,\gamma)=\sigma(\beta,\gamma)\sigma(\alpha,\beta\gamma)$ and $\sigma(\alpha,e)=\sigma(e,\alpha)=1$ for $\alpha,\beta,\gamma\in \Gamma$. 
Following \cite{VaDo2024,Enstad22}, we define
\begin{enumerate}
\item the {\it $\sigma$-twisted left group von Neumann algebra} $\mathcal{L}(\Gamma,\sigma)$ is the weak operator closed algebra generated by $\{\lambda_{\sigma}(\gamma)|\gamma\in \Gamma\}$.
    \item the {\it $\sigma$-twisted right group von Neumann algebra} $\mathcal{R}(\Gamma,\sigma)$ is the weak operator closed algebra generated by $\{\rho_{\sigma}(\gamma)|\gamma\in \Gamma\}$. 
\end{enumerate}
It is known that $\mathcal{R}(\Gamma,\overline{\sigma})$ is the commutant of $\mathcal{L}(\Gamma,\sigma)$ on $l^2(\Gamma)$  (see \cite[\S 1]{Kleppner62}). 
When $\sigma$ is trivial, $\mathcal{L}(\Gamma,\sigma)$ reduces to $\mathcal{L}(\Gamma)$. 
Thus, if $H^2(\Gamma;\mathbb{T})$ is trivial, all these $\mathcal{L}(\Gamma,\sigma)$ are isomorphic to the untwisted group von Neumann algebra $\mathcal{L}(\Gamma)$.

Observe that there is a natural trace $\tau\colon \mathcal{L}(\Gamma,\sigma)\to \mathbb{C}$  given by
\begin{center}
    $\tau(x)=\langle x\delta_e,\delta_e\rangle_{l^2(\Gamma)}$.
\end{center}
It gives an inner product on $\mathcal{L}(\Gamma,\sigma)$ defined by $\langle x,y \rangle_{\tau}=\tau(xy^*)$ for $x,y\in \mathcal{L}(\Gamma,\sigma)$. 
The completion of $\mathcal{L}(\Gamma,\sigma)$ with respect to the inner product $\langle -,- \rangle_{\tau}$ is exactly $l^2(\Gamma)$. 
More generally, for a tracial von Neumann algebra $M$ with the trace $\tau$, we consider the GNS representation of $M$ on the Hilbert space constructed from the completion of $M$ with respect to the inner product $\langle x,y\rangle_{\tau}=\tau(xy^*)$. 
The underlying space will be denoted by $L^2(M,\tau)$, or simply $L^2(M)$.

Suppose $\pi\colon M\to B(H)$ is a normal \footnote{We call such a representation $\pi\colon M\to B(H)$ {\it normal} if $\pi$ is continuous with respect to the weak topology on $M$.} unital representation of $M$ with both $M$ and $H$ separable. 
There exists an isometry $u\colon H\to L^2(M)\otimes l^2(\mathbb{N})$ such that for all $x\in M$, 
$u\circ\pi(x)=(\lambda(x)\otimes\id_{l^2(\mathbb{N})} )\circ u$, $\forall x\in M$, i.e., the following diagram commutes 
\begin{center}
\begin{tikzcd}
	H && H \\
	{L^2(M)\otimes l^2(\mathbb{N})} && {L^2(M)\otimes l^2(\mathbb{N})}
	\arrow["{\pi(x)}", from=1-1, to=1-3]
	\arrow["u", from=1-1, to=2-1]
	\arrow["u", from=1-3, to=2-3]
	\arrow["{\lambda(x)\otimes {\rm id}}", from=2-1, to=2-3]
\end{tikzcd},
\end{center}
where $\lambda\colon M\to B(L^2(M))$ denotes the left multiplication.  
Then we obtain a projection $p=uu^*$ such that
\begin{center}
    $p\colon L^2(M)\otimes l^2(\mathbb{N})\to p(L^2(M)\otimes l^2(\mathbb{N}))\cong H$, 
\end{center}
which intertwines with the $M$-actions. 
Let $M'$ denotes the commutant of $M$, i.e., $M'=\{T\in B(H)\mid Tx=xT,~\forall x \in M\}$ for an $M$-module $H$. 
We have the following result (see \cite[Proposition 8.2.3]{APintrII1}). 

\begin{proposition}\label{ptrdim}
The correspondence $H\mapsto p$ above defines a bijection between the set of equivalence classes of left $M$-modules and the set of equivalence classes of projections in $(M'\cap B(L^2(M)))\otimes B(l^2(\mathbb{N}))$. 
\end{proposition}

The {\it von Neumann dimension} of the $M$-module $H$ are defined to be $(\tau\otimes \Tr)(p)$ and denoted by $\dim_M(H)$, which takes its value in $[0,\infty]$. 
We have: 
\begin{enumerate}
    \item $\dim_M(\oplus_i H_i)=\sum_i \dim_M(H_i)$. 
    \item $\dim_M(L^2(M))=1$.
\end{enumerate} 
Note that $\dim_M(H)$ depends on the trace $\tau$. 
Moreover, if $M$ is a finite factor, i.e., $Z(M)\cong\mathbb{C}$, there is a unique normal tracial state and we further have: 
\begin{enumerate}
\setcounter{enumi}{2}
    \item $\dim_M(H)=\dim_M(H')$ if and only if $H$ and $H'$ are isomorphic as $M$-modules (provided $M$ is a factor).  
\end{enumerate}
When $M$ is not a factor, there is a $Z(M)$-valued trace which determines the isomorphism class of an $M$-module (see \cite{Bek04}).  

Now we consider the case that $\Gamma$ is a discrete subgroup of a locally compact unimodular type I group $G$.  
Let $\mu$ be a Haar measure of $G$. A measurable set $D\subset G$ is called a {\it fundamental domain} for $\Gamma$ if $D$ satisfies $\mu(G\backslash \cup_{\gamma\in\Gamma}\gamma D)=0$ and 
$\mu(\gamma_1 D\cap \gamma_2 D)=0$ if $\gamma_1\neq \gamma_2$ in $\Gamma$. 
We always assume $\Gamma$ is a lattice, i.e., $\mu(D)<\infty$.
The measure $\mu(D)$ is called {\it covolume} of $\Gamma$ and will be denoted by $\covol(\Gamma)$. 
Note that the covolume depends on the Haar measure $\mu$.




\begin{theorem}\label{tAS}
Let $G$ be a unimodular group and $\Gamma$ be a lattice in $G$. 
Suppose $\pi$ is a square-integrable representation of $G$. 
We have
\begin{center}
    $\dim_{\mathcal{L}(\Gamma)}\pi=\mu(\Gamma\backslash G)\cdot d(\pi)$.
\end{center}
\end{theorem}
This result was first proved in \cite[\S 3]{AS77} for $G$ being semi-simple simply-connected Lie groups without compact factors and $\Gamma$ being comcompact. 
The proof actually applies to this more general case (see \cite[Theorem 3.3.2]{GHJ}). 
We briefly review the generalization of Theorem \ref{tAS} for the twisted group von Neumann algebras and tempered representations in \cite{Y25AS}.

Recall that $(M,\tau)$ is a tracial von Neumann algebra. 
The following notion generalizes the definition of $\Gamma$-dimensions (see \cite[\S 1]{LuckL2book}). 
\begin{definition}[von Neumann densities]\label{dvnden}
Let $(X,\nu)$ be a measure space and $\{H_x\}_{x\in X}$ be a field of Hilbert spaces over $X$ such that there exists a unique constant $C>0$ and for any measurable $Y\subset X$, 
\begin{equation} 
H_Y=\int_{Y}^{\oplus }H_{x}d\nu(x) \text{~is an~} M\text{-module}  \text{~such that~}
    \dim_{M}H_Y=C\cdot \nu(Y). 
\end{equation}
We call $C\cdot d\nu(x)$ the \textbf{von Neumann density} of $H_x$ over $M$ and denote it by $d_{M}(H_x)$. 

If $M=\mathcal{L}(\Gamma)$ or $\mathcal{L}(\Gamma,\omega)$ for a countable discrete group $\Gamma$ and a $2$-cocycle $\omega$, we simply denote $d_{M}(H_x)$ by $ d_{\Gamma}(H_x)$ or $d_{\Gamma,\omega}(H_x)$ and call it \textbf{$\Gamma$-density} or \textbf{$(\Gamma,\omega)$-density} of $H_x$.    
\end{definition}

Although we do not assume locally that $H_{x}$ is a $M$-module for each $x$, one can show 
$d_{M}(H_x)=\dim_M H_x$ if $H_x$ is a $M$-module and $\nu(\{x\})=1$. 
In this case, we can further show
\begin{center}  $\dim_{M}\int_{Y}^{\oplus}H_{x}d\nu(x)=\int_{Y}d_{M}H_{x}d\nu(x)$, 
\end{center}
which is proved in \cite[Proposition 2.1]{Y25AS}. 

Let $\nu_{G,\sigma}$ be the Plancherel measure on $\Pi(G,\sigma)$, the set of equivalence classes of $\sigma$-projective irreducible representations of $G$ (see Theorem \ref{tprojPthm}) for a 2-cocycle $\sigma$ of $G$ . 
For $X\subset\Pi(G,\sigma)$ such that $\nu_{G,\sigma}(X)<\infty$, we let
\begin{center}
    $H_X=\int_X^{\oplus} H_{\pi}d\nu_{G,\sigma}(\pi)$,
\end{center}
which is a module over $\mathcal{L}(G,\sigma)$ and $\mathcal{L}(\Gamma,\sigma)$ (the $\sigma$-twisted left group von Neumann algebra of $G,\Gamma$) respectively). 
The following result is the generalization of Theorem \ref{tAS} (see \cite[Theorem 3.5]{Y25AS}). 

\begin{theorem}\label{tdimmeas}
Let $\Gamma$ be a lattice of $G$.   
We have 
\begin{equation}\label{eASf1}
   \dim_{\mathcal{L}(\Gamma,\sigma)}H_X=\mu(\Gamma\backslash G)\cdot \nu_{G,\sigma}(X),
\end{equation}
or, equivalently, 
\begin{equation}\label{eASf1a}
  d_{\Gamma,\sigma}(H_{\pi})=\mu(\Gamma\backslash G)\cdot d\nu_{G,\sigma}(\pi)
\end{equation}
for a $\sigma$-projective tempered irreducible representation $(\pi,H_{\pi})$ of $G$. 
\end{theorem}

Here we provide examples of tempered representations and projective tempered representations. 

\begin{example}
\begin{enumerate}
    \item Let $G=\SL(2,\mathbb{R})$. 
We take $d\mu(g)=d\lambda(t) y^{-2}dx dy$ as the Haar measure on $G$. 
For $\Gamma=\SL(2,\mathbb{Z})$, one has $\mu(\Gamma\backslash G)=\frac{\pi}{3}$. 
For $k\geq 2$, we let $(\pi_k,H_{k})$ be the holomorphic discrete series representation of $G$ whose minimal $\SO(2)$-weight is $k$. 
We have the formal degree $d(\pi_k)=\frac{k-1}{4\pi}$. 
We take $\Gamma=\SL(2,\mathbb{Z})$. 
We have
\begin{center}
    $d_{\Gamma}(H_k)=\dim_{\mathcal{L}(\Gamma)}H_k=\frac{k-1}{12}$
\end{center}
by Theorem \ref{tAS} (see also \cite[Example 3.3.4]{GHJ} and \cite[\S 6]{Ra14}). 

Let $\{(\pi^{\pm}_{it},H_{it}^{\pm})|t>0\}$ be the principal series representations of $G$, which are given by the parabolic induction of characters $\chi_t^{\pm}\begin{pmatrix}
a & b\\
0 & a^{-1}
\end{pmatrix}=\varepsilon^{\pm}(a)|a|^{it}$, where $\varepsilon^{+}(a)=1$
and $\varepsilon^{-}(a)={\rm sign}(a)$.  
They are tempered representations that are not discrete series. 
Then, by the explicit Plancherel measure of $\SL(2,\mathbb{R})$ in \cite[\S 7.6]{Fo2}, we have
\begin{center}
    $d_{\Gamma}(H_{it}^{+})=\frac{t}{24}\tanh\frac{\pi t}{2}dt$ and $d_{\Gamma}(H_{it}^{-})=\frac{t}{24}\coth\frac{\pi t}{2}dt$,
\end{center}
where $dt$ denotes the Lebesgue measure of the real parameter $t>0$.

\item For $G=\PSL(2,\mathbb{R})$ and $\{(\pi_r,H_r)\}_{r>1}$ be the family of projective irreducible representation of $G$ (see \cite[Definition 1.1]{Ra98}). 
Let $\Gamma=\PSL(2,\mathbb{Z})$. 
We know that $H^2(\PSL(2,\mathbb{Z});\mathbb{T})=1$ and $\mathcal{L}(\Gamma,\sigma)\cong \mathcal{L}(\Gamma)$ for any $2$-cocycle $\sigma$.
By \cite[Theorem 3.2]{Ra98}, we have
\begin{center}
    $d_{\Gamma,\sigma}(H_r)=d_{\Gamma}(H_r)=\dim_{\mathcal{L}(\Gamma)}H_r=\frac{r-1}{12}$. 
\end{center}
\end{enumerate}
\end{example}

\section{The Jacquet-Langlands correspondence of Plancherel measures}\label{sJLmeas}

In this section, we give explicit Haar measures on the groups $\GL(n)$ and its inner forms over local fields. 
We show that the local Jacquet-Langlands correspondence preserves the associated Plancherel measures. 
Then we prove the compatibility of the local Plancherel measures and the global Tamagawa measure.

\subsection{The Haar measures on inner forms of $\GL(n)$: the non-archimedean case}\label{ssHaarp}

We fix a non-archimedean place $v\in V_f$ throughout this section. 
Consider the $F_v$-algebra $A_v\cong \Mt(n_v,D_v)$ and its multiplicative group $A_v^{\times}\cong \GL(n_v,D_v)$. 

There is a valuation map on $D_v$ given as $w_v=v \circ \Nm_{D_v/F_v}\colon D_v\to \mathbb{Z}$. 
Let $O(D_v)$ be the maximal $\mathcal{O}_v$-order of $D_v$ and $J(D_v)$ be the maximal ideal of $O(D_v)$. 
It is known that (see \cite[Theorem 12.8 and 13.2]{Rei03})
\begin{center}
    $O(D_v)=\{x\in D_v|w_v(x)\geq 0\}$,  $J(D_v)=\{x\in D_v|w_v(x)> 0\}$, 
\end{center}
which are unique up to isomorphism. 

For the $F_v$-algebra $A_v\cong \Mt(n_v,D_v)$, we recall that $O(A_v)$ is its maximal $\mathcal{O}_v$-order and let $J(A_v)$ be the maximal two-sided ideal of $O(A_v)$ (see Section \ref{ssdivnumfield}).  
It is known that (see \cite[Theorem 17.3]{Rei03})
\begin{center}
    $O(A_v)=\Mt(n_v,O(D_v))$,  $J(A_v)=J(D_v)O(A_v)$. 
\end{center}
The additive Haar measure $\mu^{+}_v$ on $A_v$ is defined to be the one normalized in the sense that its maximal order has volume $1$, i.e., $\mu^{+}_v(O(A_v))=1$.  
The multiplicative Haar measure $\mu^{\times}_v$ is defined by
\begin{center}
    $d \mu^{\times}_v(x)=(1-1/q_v)^{-1} \frac{d \mu^{+}_v(x)}{\|x\|_{A_v}}$,
\end{center}
where $\|x\|_{A_v}$ is the norm on $A_v$ defined as $\|x\|_{A_v}=\|\Nm_{A_v/F_v}(x)\|_{F_v}$. 

\begin{lemma}\label{lvolmaxorder}
$\mu^{\times}_v(O^{\times}(A_v))=(1-1/q_v)^{-1}\cdot \prod_{i=1}^{n_v}(1-(q_v^{-d_v})^i)$. 
\end{lemma}
\begin{proof}
Note that $O(D_v)/J(D_v)$ is a skewfield over $k(v)$ and hence a finite field (see \cite[Theorem 13.2 and Theorem 7.24]{Rei03}).
It is known that $\dim_{k(v)}O(D_v)/J(D_v)=d_v$ (see \cite[Theorem 14.3]{Rei03}). 
Consider the short exact sequence
\begin{center}
    $1\to 1+J(D_v)\Mt(n_v,O(D_v))\to \GL(n_v,O(D_v))\to \GL(n_v,O(D_v)/J(D_v))\to 1$. 
\end{center}
Since $O(A_v)/J(D_v)\Mt(n_v,O(D_v))=\Mt(n_v,\mathbb{F}_{q_{v}^{d_v}})$, 
we have
\begin{equation*}
\begin{aligned}
    \mu^+_{v}(1+J(D_v)\Mt(n_v,O(D_v)))&=\mu^+_{v}(J(D_v)\Mt(n_v,O(D_v)))\\
    &=\frac{1}{\# \Mt(n_v,\mathbb{F}_{q_{v}^{d_v}})}\cdot \mu^{+}_v(O(A_v))\\
    &=q_v^{-d_v n_v^2}. 
\end{aligned}
\end{equation*}
Since $\|x\|_{A_v}=1$ for $x\in \GL(n_v,O(D_v))$, we have
\begin{equation*}
\begin{aligned}
\mu^{\times}_v(O(A_v))&=(1-1/q_v)^{-1}\cdot \mu^{+}_v(1+J_v\cdot \Mt(n_v,O(D_v))\cdot |\GL(n_v,O(D_v)/J(D_v))| \\
&=(1-1/q_v)^{-1}\cdot q_v^{-d_v n_v^2}\cdot \prod_{i=0}^{n_v-1}((q_v^{d_v})^{n_v}-(q_v^{d_v})^i)\\
&=(1-1/q_v)^{-1}\cdot \prod_{i=1}^{n_v}(1-(q_v^{-d_v})^i). 
\end{aligned}
\end{equation*}
\end{proof}

In particular, for the split case when $r_v=d$, $d_v=1$ and $O(A_v)^{\times}\cong \GL(nd,\mathcal{O}_v)$, we have
\begin{center}
 $\mu^{\times}_v(\GL(nd,\mathcal{O}_v))=(1-1/q_v)^{-1}\prod_{i=1}^{nd}(1-q_v^{-i})$.    
\end{center}

Now we define the local Tamagawa measure on $A_v^{\times}\cong \GL(n_v,D_v)$ as 
\begin{center}
    $\mu_v=\N(\disc_{\mathbb{Z}_p}(O(A_v)))^{-1/2}\cdot \mu^{\times}_v$. 
\end{center}
It is self-dual to the standard unitary character $\psi_{A_v}\colon A_v\to  F_v$ constructed as follows (see \cite[Proposition 29.7.7]{Voight21}):
\begin{enumerate}
    \item For $x\in \mathbb{Q}_p$, we let $\langle x\rangle_{\mathbb{Q}_p}\in \mathbb{Q}$ such that $0\leq \langle x\rangle_{\mathbb{Q}_p}<1$ and $x-\langle x\rangle_{\mathbb{Q}_p}\in\mathbb{Z}_p$.
    \item For $x \in F_v$ with $F_v/\mathbb{Q}_p$ a finite extension, we let $\langle x\rangle_{F_v}=\langle \Tr_{F_v/\mathbb{Q}_p} (x)\rangle_{\mathbb{Q}_p}$ and $\psi_{F_v}(x)=\exp(2\pi i \langle x\rangle_{F_v})$
    \item For $x\in A_v$, we let $\psi_{A_v}(x)=\psi_{F_v}(\tr(x))$.  
\end{enumerate}

The following two results are needed for the normalization of the Haar measure and the Plancherel measures. 
Observe $\disc_{\mathbb{Z}_p}(\mathcal{O}_v)$ is the local discriminant $d(F_v)$ if $v$ is above $p$. 
\begin{lemma}\label{lpNdisc}
$\N(\disc_{\mathbb{Z}_p}(O(A_v)))=d(F_v)^{n^2d^2}q_v^{d_{v}(d_{v}-1)n_v^2}$
\end{lemma}
\begin{proof}
Observe that $O(A_v)$ is a lattice in $A_v$ of free rank $n_v^2 d_v^2=n^2d^2$ and $O(A_v)=\Mt(n_v,O(D_v))$. 
We have
\begin{equation*}
\begin{aligned}
    \N(\disc_{\mathbb{Z}_p}(O(A_v)))&=\N(\disc_{\mathbb{Z}_p}(\mathcal{O}_v))^{n^2d^2}\N(\disc_{\mathcal{O}_v}(O(A_v)))\\
    &=\N(\disc_{\mathbb{Z}_p}(\mathcal{O}_v))^{n^2d^2}\cdot \N(\disc_{\mathcal{O}_v}(O(D_v)))^{n_v^2}.
\end{aligned}
\end{equation*}
By \cite[Theorem 14.9]{Rei03}, $\disc_{\mathcal{O}_v}(O(D_v))=P_{v}^{d_{v}(d_{v}-1)}$. 
Thus $\N(\disc_{\mathcal{O}_v}(O(D_v)))=q_{v}^{d_{v}(d_{v}-1)}$ and $\N(\disc_{\mathbb{Z}_p}(O(A_v)))=\N(\disc_{\mathbb{Z}_p}(\mathcal{O}_v))^{n^2d^2}q_v^{d_{v}(d_{v}-1)n_v^2}$.  
\end{proof}

Combining Lemma \ref{lvolmaxorder} and Lemma \ref{lpNdisc}, we obtain the volume of $O^{\times}(A_v)$ under the local Tamagawa measure.
\begin{corollary}\label{cTamaMaxod}
$\mu_v(O^{\times}(A_v))=d(F_v)^{-\frac{1}{2}n^2d^2}q_v^{-\frac{1}{2}d_{v}(d_{v}-1)n_v^2}(1-1/q_v)^{-1}\prod_{i=1}^{n_v}(1-(q_v^{-d_v})^i)$.
\end{corollary}

\subsection{The correspondence of non-archimedean Plancherel measures}\label{ssJLp}

We simply write $G$ for $G_v=\GL(n,F_v)$ and write $G'$ for $G'_v=\GL(n_v,D_v)$ for the rest of this section.  
We discuss the tempered dual and the canonical Harish-Chandra measure on them as described in \cite{Waldsp03}. 

Let $M$ be a Levi subgroup and $\im X(M)$ be the group of unitary unramified characters of $M$. 
Denote by $\Pi_2(M)$ the set of equivalence classes of the square-integrable irreducible representations (modulo the center) of $M$. 
The group $\im X(M)$ has a natural action on $\Pi_2(M)$ given as $\chi\circ \pi=\chi\pi$ for $\pi \in \Pi_2(M), \chi\in \im X(M)$. 
Let $P(G)$ be the set of semi-standard parabolic subgroups of $G$. 
For $P\in P(G)$, we let $P=MU$ be the Levi decomposition. 
We define
\begin{center}
    $\Theta=\{(\mathcal{O},P)|P=MU\in P(G), \mathcal{O}\subset \Pi_2(M)\text{ is an } \im X(M)\text{-orbit}\}$. 
\end{center}
Let $W$ be the Weyl group of $G_v$ and $W_M$ be the Weyl group of $M$.
We define $W(G|M)=\{s\in W|s\cdot M=M\}/W_M$ and $\Stb(\mathcal{O},M)=\{s\in W(G|M)|s\mathcal{O}=\mathcal{O}\}$. 

For $\omega \in \Pi_2(M)$, we let $\Ind^G_{P}(\omega)$ be the normalized induced representation from $\omega$. 
Thus we get a map $(\mathcal{O},P=MU)\to \Pi_{\tp}(G)$ given by $\omega\mapsto \Ind^G_{P}(\omega)$. 
This gives a bijection
\begin{equation}\label{ebijetion}
    \bigsqcup\limits_{(\mathcal{O},P=MU)\in \Theta}(\mathcal{O},P=MU)/\Stb(\mathcal{O},M) \to \Pi_{\tp}(G). 
\end{equation}
Therefore $\Pi_{\tp}(G)$ can be identified with the left-hand side as a disjoint union of countably many compact orbifolds. 

Let $A_M$ be the center of $M$. 
We assign a Haar measure $\mu_{0}$ to the group $\im X(A_M)$ of unramified unitary characters on $A_M$ such that its total mass is $1$. 
Consider the map $f\colon\im X(M)\to\mathcal{O}$ given by $\chi\mapsto \omega\otimes \chi$ and the restriction map $h\colon \im X(M)\to \im X(A_m)$. 
Then we defined a measure $d \omega_v$ on $\mathcal{O}$ such that $f^{*}\omega=h^{*}\mu_{0}$ on $\im X(M)$. 

Let $E\subset \mathcal{O}$ be a fundamental set of the action of $\Stb(\mathcal{O},M)$. 
We define $d \omega_v$ to be the measure on $E$ as follows: for any $\Stb(\mathcal{O},M)$-invariant integrable function $F$ on $\mathcal{O}$, $\int_{\mathcal{O}}Fd\omega_v=|\Stb(\mathcal{O},M)|\cdot \int_{E}Fd \omega_v$. 
We extend $d \omega_v$ to $\Pi_{\tp}(G)$ via the bijection \eqref{ebijetion} and call it the {\it canonical Harish-Chandra measure} on  $\Pi_{\tp}(G)$.

For $f\in \mathcal{H}(G)$, the {\it orbital integral} of $f$ is a function on the set regular semi-simple elements of $G$ given by
\begin{center}
    $\Phi(f;g)=\int_{G/Z_G(g)}f(xgx^{-1})dx$
\end{center} 
for $g\in G^s$, where $Z_G(g)$ is the centralizer of $g$ in $G$. 
The orbit integral for $G'$ can be defined similarly. 
Recall that  $f'\in \mathcal{H}(G)$ and  $f\in \mathcal{H}(G')$ correspond to each other, denoted by $f\leftrightarrow f'$, if
\begin{enumerate}
    \item if $g\leftrightarrow g'$ for $g\in G,g'\in G'$, then $\Phi(f;g)=(-1)^{n_v-n}\cdot\Phi(f';g')$;
    \item if for $g\in G$ there does not exist $g'\in G'$ such that $g\leftrightarrow g'$, then $\Phi(f;g)=0$. 
\end{enumerate}
For $f'\in \mathcal{H}(G')$, there exists $f\in \mathcal{H}(G)$ such that $f\leftrightarrow f'$.  Conversely, for $f\in \mathcal{H}(G)$ with ${\rm supp}(f)\subset G^s$, there exists $f'\in \mathcal{H}(G')$ such that $f\leftrightarrow f'$ (see \cite[\S 2.5]{Badu2018}). 

Let $\St_G,\St_{G'}$ be the Steinberg representations of $G,G'$ respectively. 
The following result is known, whose proof is added briefly here for completeness.   
\begin{lemma}\label{lf1St}
If $f\leftrightarrow f'$, we have $\frac{f(1)}{\deg(\St_G)}=\frac{f'(1)}{\deg(\St_{G'})}$. 
\end{lemma}
\begin{proof}
For $g\in G,g'\in G'$, their conjugacy classes $C_{G}(g),C_{G'}(g')$ are closed in $G,G'$ respectively. 
For a unipotent conjugacy class $O$, we let $\Lambda_{O}(f)$ be the orbital integral over $O$, i.e., $\Lambda_{O}(f)=\int_{O}f(xgx^{-1})dx$. 
By \cite[Theorem 2.1.1]{Shalika72}, 
there exist functions $\Gamma^G_O\colon G^s\to \mathbb{R}$ that is identical on each unipotent $O$, which satisfy
\begin{equation}\label{eshalika}
    \Phi(f,g)=\sum_{O{\rm~unipotent}} \Gamma_{O }^{G}(g)\cdot \Lambda_{O}(f)
\end{equation} 
if $g$ is sufficiently close to $1$.  
Moreover, $\Lambda_{1}(f)=f(1)$ and all these also hold for $G'$. 
By \cite[Theorem]{Rog80}, we know $\Gamma_1^G=\frac{(-1)^{n}}{\deg(\St_G)}$ and $\Gamma_1^{G'}=\frac{(-1)^{n_v}}{\deg(\St_{G'})}$. 
It then follows the comparison of Equation \eqref{eshalika} of $f,f'$ at $g=1,g'=1$. 
\end{proof}

Please note that while the relation $f\leftrightarrow f$ and $\deg(\St_G),\deg(\St_G')$ depend on the choice of the Haar measure on $G,G'$, these dependences cancel out in the equation of the two quotients. 
 
We take the Haar measures $\mu_v$ on $\GL(nd,F_v),\GL(n_v,D_v)$ as above and let $\nu_v,\nu_v'$ be the Plancherel measure on $\Pi_{\tp}(\GL(nd,F_v)),\Pi_{\tp}(\GL(n_v,D_v))$ respectively. 
We denote by $d \nu_v,d \nu'_v$ the Plancherel densities with respect to the canonical Harish-Chandra measure $d \omega_v$.
Note that the local Jacquet-Langlands correspondence $\JL_v$ gives a homeomorphism from a subset $Y\subset \Pi_{\tp}(\GL(nd,F_v))$ to $\Pi_{\tp}(\GL(n_v,D_v))$.  
The following result is essentially inspired by Aubert and Plymen (see \cite[Theorem 7.2]{AubtPlm05}).

\begin{proposition}\label{pratiopPd}
If $\JL_v(\pi)=\pi'$ with $\pi,\pi'$ tempered and nontrivial,  
we have
\begin{center}
    $d \nu_v (\pi)=d \nu'_v(\pi')$,
\end{center}
if the Haar measures are taken to be the local Tamagawa measures $\mu_v$. 
\end{proposition}
\begin{proof}
Assuming $f\leftrightarrow f'$ and $\pi'$ is an irreducible representation of $G'$ induced from some Levi subgroup $M'=\prod \GL(m_i,D_v)$ with $\sum m_i=n_v$, we know
\begin{equation}\label{ethetaf}
    \theta_{\pi'}(f')=\theta_{\JL_{v}(\pi')}(f)
\end{equation}
by \cite[Proposition 3.6]{Badu2003}.   
It is also proved that $\theta_{\pi}(f)=0$ if $\pi$ is not an induced representation comes from a Levi subgroup $M=\prod \GL(n_i,F_v)$ such that all $n_i$ are divided by $d$ (see \cite[\S 3]{Badu2003}).

We compare the Plancherel formulas of $G,G'$ at $g=g'=1$. 
By Lemma \ref{lf1St}, we obtain
\begin{equation}\label{ePformDegSt}
    \frac{\int_{Y}\theta_{\pi}(f)d \nu(\pi)}{\deg(\St_G)}=\frac{\int_{\Pi_{\tp}(G')}\theta_{\pi'}(f')d \nu'(\pi')}{\deg(\St_{G'})}. 
\end{equation}
By \cite[Lemma 3.4]{Badu2003}, the characters $\theta_{\pi}(f^*)$ spans the space of compactly supported functions on $\Pi_{\tp}(G')$ (here $f^*(g)=f(g^{-1})$). 
By Equation \eqref{ethetaf} and \eqref{ePformDegSt}, the measures $\frac{d\nu|_{Y}}{\deg(\St_G)}$ and $\frac{d\nu'}{\deg(\St_{G'})}$ coincides on a dense subset of integrable function on $\Pi_{\tp}(G')\cong Y\subset \Pi_{\tp}(G)$. 
Thus we obtain
\begin{equation}\label{edegPmeas}
    \frac{d\nu|_{Y}}{\deg(\St_G)}=\frac{d\nu'}{\deg(\St_{G'})}.
\end{equation}
It then suffices to discuss $\deg(\St_G),\deg(\St_{G'})$. 

By \cite[\S 2]{CMS90}, if the Haar measures of $\GL(n,F_v)$ and $\GL(n_v,D_v)$ are normalized such that the maximal compact subgroups $\GL(n,\mathcal{O}_v)$ and $\GL(n_v,O(D_v))$ are of total mass $1$, we have
\begin{equation}\label{eolddeg}
    \deg(\St_G)=\frac{1}{n}\prod_{i=1}^{n-1}(q^{i}-1),\deg(\St_{G'})=\frac{1}{n}\prod_{i=1}^{n_v-1}(q^{d_{v}i}-1)
\end{equation}
We now take the Haar measures $\mu_v$ above. 
By Corollary \ref{cTamaMaxod}, the volumes of these maximal compact subgroups are given as follows.
\begin{enumerate}
\item If $A_v\cong \Mt(nd,F_v)$, we have $O^{\times}(A_v)\cong \GL(nd,\mathcal{O}_v)$ and 
    \begin{center}
        $\mu_v(\GL(nd,\mathcal{O}_v))=d(F_v)^{-\frac{1}{2}n^2d^2}(1-1/q_v)^{-1}\prod_{i=1}^{nd}(1-q_v^{-i})$. 
    \end{center}
\item If $A_v\cong \Mt(n_v,D_v)$ and $\dim_{F_v}D_v=d_v^2$, we have $O^{\times}(A_v)\cong \GL(n_v,O(D_v))$ and
    \begin{center}
        $\mu_v(\GL(n_v,O(D_v)))=d(F_v)^{-\frac{1}{2}n^2d^2}q_v^{-\frac{1}{2}d_{v}(d_{v}-1)n_v^2}(1-1/q_v)^{-1}\prod_{i=1}^{n_v}(1-(q_v^{-d_v})^i)$; 
    \end{center}
    
\end{enumerate}
Thus we get
\begin{equation}\label{enewquoMeas}
    \frac{\mu_v(\GL(nd,\mathcal{O}_v))}{\mu_v(\GL(n_v,O(D_v)))}=q_v^{\frac{1}{2}d_{v}(d_{v}-1)n_v^2}\prod_{i=1,d_v\nmid i}^{nd}(1-q^{-i}).
\end{equation}
By Equation \eqref{edegPmeas}, the quotient of the two Plancherel densities is given by the quotient of formal degrees of $\St_{G},\St_{G'}$, which can be obtained from Equation \eqref{eolddeg} and \eqref{enewquoMeas} as follows. 
\begin{equation*}
\begin{aligned}
    \frac{d \nu_v'(\pi_v')}{d \nu_v(\pi_v)}&=\prod_{i=1,d_v\nmid i}^{nd}(q_v^i -1)^{-1} \cdot \prod_{i=1,d_v\nmid i}^{nd}(1-q^{-i})\cdot q_v^{\frac{1}{2}d_{v}(d_{v}-1)n_v^2}\\
    &=\prod_{i=1,d_v\nmid i}^{nd}\frac{1-q_v^{-i}}{q_v^i-1}\cdot q_v^{\frac{1}{2}d_{v}(d_{v}-1)n_v^2} =\prod_{i=1,d_v\nmid i}^{nd}q_{v}^{-i}\cdot q_v^{\frac{1}{2}d_{v}(d_{v}-1)n_v^2}\\
    &=\frac{\prod_{j=1}^{n_v}q_{v}^{d_v\cdot j}}{\prod_{i=1}^{nd}q_{v}^{i}}\cdot q_v^{\frac{1}{2}d_{v}(d_{v}-1)n_v^2}=q_v^{\frac{1}{2}d_{v}n_{v}(n_{v}+1)}\cdot q_v^{-\frac{1}{2}n_{v}d_{v}(n_{v}d_{v}+1)}\cdot q_v^{\frac{1}{2}d_{v}(d_{v}-1)n_v^2}\\&=
    q_v^{\frac{1}{2}n_{v}d_{v}(n_{v}+1+n_{v}d_{v}-n_{v}-n_{v}d_{v}-1)}=1.
\end{aligned}
\end{equation*}
\end{proof}

\subsection{The Haar measures on inner forms of $\GL(n)$: the arcimedean case}\label{ssHaarR}

The inner forms of $\GL(n,\mathbb{R})$ are given by the multiplicative group of simple $\mathbb{R}$-algebra of the same dimension. 
We let $A$ be a finite-dimensional simple algebra over $\mathbb{R}$. 
It is known that $A\cong \Mt(n,\mathbb{R})$ or $A\cong \Mt(n,\mathbb{H})$, where $\mathbb{H}$ is the Hamilton quaternion algebra defined as the $4$-dimensional $\mathbb{R}$-algebra with basis $1,i,j,k$ subject to the relations $i^2=j^2=k^2=-1$ and $ij=-ji=k,jk=-kj=i,ki=-ik=j$. 

Suppose $\dim_{\mathbb{R}}A=m$ and we choose an $\mathbb{R}$-basis $e_1,\cdots,e_m$ for $A$. 
We define the (additive) Haar measure $\mu_{\mathbb{R}}^{+}$ on $A$ to be given as
\begin{center}
    $d x=|d(e_1,\cdots,e_m)|^{1/2}d x_1\cdots d x_m$,
\end{center}
where $x=\sum_{1\leq i\leq m}x_i e_i$, $d(\cdots)$ is the discriminant of a basis and $d x_i$ is the Lebesgue measure of $\mathbb{R}$.  
Then the multiplicative Haar measure on $A^{\times}$ is defined as
\begin{center}
    $d \mu_{\mathbb{R}}(x)=\frac{d \mu_{\mathbb{R}}^{+}(x)}{\|x\|_{D}}$,
\end{center}
where $\|x\|_{A}=|\Nm_{A/\mathbb{R}}(x)|$ for $x\in A^{\times}$.  
By extending the notions in \cite[\S 7.5]{McRd219}, 
we call it the {\it local Tamagawa measure} of $A^{\times}$ and will be taken as the Haar measures on $A^{\times}$. 

Let us consider the square-integrable irreducible representations of $\GL(2,\mathbb{R})$ and of $\mathbb{H}^{\times}$, which are the multiplicative groups of a four-dimensional $\mathbb{R}$-algebra $D$. 
\begin{enumerate}
    \item For $D=\Mt(2,\mathbb{R})$, we know $D^{\times}=\GL(2,\mathbb{R})=\mathbb{R}^{\times}_{>0}\cdot \SL^{\pm}(2,\mathbb{R})$. 
Let $H_n^+$ and $H_n^-$ be the holomorphic and anti-holomorphic square-integrable representations of $\SL(2,\mathbb{R})$ respectively (see \cite{Kn77GL2}).  
The square-integrable irreducible representations of $\GL(2,\mathbb{R})$ are given by
\begin{center}
    $H_n(\omega):=\omega\otimes \Ind_{\SL(2,\mathbb{R})}^{\SL^{\pm}(2,\mathbb{R})}H_n^+=\omega\otimes(H_n^+\oplus H_n^-)$, $n\geq 1$,
\end{center}
where $\omega$ is a character of $\mathbb{R}^{\times}_{>0}$.
\item For $D=\mathbb{H}$, we know $D^{\times}=\mathbb{H}^{\times}=\mathbb{R}^{\times}_{>0}\cdot \mathbb{H}^1$, where  $\mathbb{H}^1=\{x\in \mathbb{H}|\Nm(x)=1\}\cong\SU(2)$. 
Thus the irreducible representations of $\mathbb{H}^{\times}$ are given by
\begin{center}
    $V_n(\omega)=\omega\otimes \mathbb{C}^n$,
\end{center}
where $\omega$ is a character of $\mathbb{R}^{\times}_{>0}$ and $\mathbb{C}^n$ is the unique $n$-dimensional irreducible representation of $\SU(2)$.
\end{enumerate}

The local Jacquet-Langlands correspondence for $\GL(2,\mathbb{R})$ is described as follows (see \cite[\S 1]{KnRg97tfJL}):
\begin{center}
   $\JL_{\mathbb{R}}\colon \Pi_2(\GL(2,\mathbb{R})) \to \Pi(\mathbb{H}^{\times})$ given by $\JL_{\mathbb{R}}(H_n(\omega))=V_n(\omega)$. 
\end{center} 
Noting $\mathbb{H}^{\times}$ is compact modulo its center $\mathbb{R}$, we identify $\Pi(\mathbb{H}^{\times})$ and $\Pi_2(\mathbb{H}^{\times})$.
Thus the formal degree of $\pi \in \Pi(\mathbb{H}^{\times})$ is given by $\frac{\dim_{\mathbb{C}}\pi}{\mu_{\mathbb{R}}(\mathbb{H}^{\times}/\mathbb{R}^{*})}$, where we treat $\mathbb{H}^{\times}/\mathbb{R}^{*}$ as a fundamental domain of $\mathbb{H}^{\times}$ with respect to the action of $\mathbb{R}^*$.

\begin{lemma}\label{larc2H}
$\JL_{\mathbb{R}}$ preserves the formal degrees. 
In particular, the formal degrees of $H_n(\omega)$ and $V_n(\omega)$ are $\frac{n}{2\pi^2}$. 
\end{lemma}
\begin{proof}
Observe $\PSL(2,\mathbb{R})=\mathcal{H}\times \PSO(2)$, where $\mathcal{H}=\{x+iy|x\in\mathbb{R},y\in \mathbb{R}_{>0}\}$ is the upper half-plane. 
If we take $y^{-2}dxdy$ for the measure of $\mathcal{H}$ and normalized measure for $\PSO(2)$, the formal degree of $H^+_n$ is $\frac{n}{4\pi}$ by \cite[\S 3.3.d]{GHJ}. 
Thus, for the Haar measure $\mu_{\mathbb{R}}$ above, $\PSO(2)$ is of total mass $\pi$ and the formal degree of $H_{n}^+$ is $\frac{n}{4\pi^2}$. 
Hence the formal degree of $H_n(\omega)$ is $\frac{n}{2\pi^2}$. 

For $\mathbb{H}^\times$, we know the $\mu_{\mathbb{R}}(\mathbb{H}^1)=4\pi^2$ (see \cite[Lemma 29.5.9]{Voight21}). 
Observe $\mathbb{H}^{\times}/\mathbb{R}^{*}=\SU(2)/\{\pm I\}$, we have $\mu_{\mathbb{R}}(\mathbb{H}^{\times}/\mathbb{R}^{*})=2\pi^2$ where $\mathbb{H}^{\times}/\mathbb{R}^{*}\subset \mathbb{H}^{\times}$ is a fundamental domain for the $\mathbb{R}^*$-action. 
Thus the formal degree of $V_n(\omega)$ is $\frac{\dim_{\mathbb{C}}{V_n(\omega)}}{\mu_{\mathbb{R}}(\mathbb{H}^{\times}/\mathbb{R}^{*})}=\frac{n}{2\pi^2}$. 
\end{proof}

\subsection{The correspondence of archimedean Plancherel measures}\label{ssJLR}

We first give some general results of a real reductive group and its inner forms. 
Then we consider the inner forms of $\GL(n,\mathbb{R})$.

Let $\mathbf{G}$ be a connected reductive group defined over $\mathbb{R}$. 
Let $\mathbf{G}'$ be an inner form of $G$, i.e., there exists an isomorphism $\psi\colon \mathbf{G} \to \mathbf{G}'$ such that $\overline{\psi}\psi^{-1}$ is an inner automorphism of $\mathbf{G}'$. 
We assume as well that $\mathbf{G}'$ is quasi-split over $\mathbb{R}$. 
We know that $G=\mathbf{G}(\mathbb{R})$ is a real reductive group and $G'=\mathbf{G}'(\mathbb{R})$ is an inner form of $G$. 

Take $\omega'$ to be a left-invariant differential form of the highest degree on $\mathbf{G}'$. 
Let $dg'$ be the Haar measure on $G'$ given by $\omega'$. 
Observe that $\psi\colon \mathbf{G}\to \mathbf{G}'$  induces a map from forms on $\mathbf{G}'$ to forms on $\mathbf{G}$. 
Thus, we let $\omega$ be the image of $\omega'$, which gives the Haar measure $dg$ on $G$. 
We let $\nu_G,\nu_{G'}$ be the Plancherel measure on the unitary dual of $G,G'$ associated with $dg,dg'$, respectively. 

Let $\Phi(G)$ be the set of {\it $L$-parameters of $G$}, i.e., the set of homomophisms $\varphi\colon W_{\mathbb{R}}\to {}^{L}G$, where $W_{\mathbb{R}}=\mathbb{C}^{\times}\cup j\mathbb{C}^{\times}$ is the Weil group of $\mathbb{R}$ with $j^{2}=-1$ and $jzj^{-1}=\overline{z}$ (see \cite[\S 8]{BoAutoL}). 
Consider the embedding $\Phi(G)\to \Phi(G')$ induced by the map $\psi\colon \mathbf{G} \to \mathbf{G}'$. 
For $\varphi\in \Phi(G)$, we let $\varphi'\in \Phi(G')$ be the image of $\varphi$ under this embedding, which we will denote by $\varphi\mapsto \varphi'$.

For $\varphi\in \Phi(G)$, we let $\Pi_{\varphi}$ be the {\it $L$-packets} of admissible representations of $G$ given by the local Langlands correspondence for real reductive groups (see \cite{Llds89}). 
We define
\begin{center}
    $\chi_{\varphi}=\sum_{\pi\in \Pi_{\varphi}}\chi_{\pi}$,
\end{center}
which is the finite sum of characters of the irreducible representations in $\Pi_{\varphi}$.

Note that one element of $\Pi_{\varphi}$ is tempered if and only if all of them are tempered, in which case we call $\varphi$ tempered. 
It is also known that $\varphi'$ is tempered if and only if $\varphi$ is tempered. 
We let $\Phi_{\tp}(G),\Phi_{\tp}(G')$ be the set of tempered $L$-parameters of $G,G'$ respectively.  

We will define {\it the Plancherel measure of a tempered real $L$-packet}. 
Let $P_0$ be the minimal parabolic subgroup of $G$ and $P_0={}^{0}M_{P_0}A_{0}N_{0}$ be its Langlands decomposition (see \cite[\S 2.2.7]{Wallach1}). 
Given a parabolic pair $P={}^{0}{M}_P AN$, a unitary representation $\omega$ of $M$ and $\nu\in \mathfrak{a}^{*}$ (where $\mathfrak{a}$ is the Lie algebra of $A$), we let 
\begin{center}
    $\Ind_{P}^{G}(\omega,i\nu)=\Ind_{(^{0}{M}_P AN)}^{G}(\omega\otimes\exp(i\nu)\otimes 1)$
\end{center} 
be the induced representation of $G$.  
The tempered dual of $G$ can be described by Harish-Chandra's Plancherel formula:
\begin{equation}\label{eRPformula}
    f(g)=\sum\limits_{(P,A)\succ (P_0,A_0)}C_A\sum\limits_{\omega\in \Pi_2(^{0}{M}_P)}d(\omega)\cdot \int_{\mathfrak{a}^{*}}\theta_{\Ind_{P}^{G}(\omega,i\nu)}(R(g)f)\mu(\omega,i\nu)d\nu,
\end{equation}
where $f$ is a Schwartz function, the first summation is taken over the parabolic groups $P$ with $P\supset P_0, A\supset A_0$ and $\nu$ is a measure on $\mathfrak{a}^{*}$ and $C_A$ is a constant depending on $A$ (see \cite[\S 13.4]{Wallach2}).  

We let $\nu_G(\pi)$ denote the Plancerel density of $\pi$ given by the formula \eqref{eRPformula}.
Thus, we may define the Plancherel measure of a tempered $L$-parameter $\varphi$ as
\begin{center}
    $d\nu_{G}(\varphi)\stackrel{\rm df}{=}\sum_{\pi\in \Pi_{\varphi}}d\nu_{G}(\pi)$,
\end{center}
which is the finite sum of Plancherel measure (density) of the irreducible representations of $G$. 
The following result should be well-known to experts. 
We sketch a proof for completeness, which is mainly based on the work of Shelstad (see \cite{Shelstad1979char,ShelstadLindis}).

\begin{theorem}\label{tshelsinn}
For a real reductive $G$ and its quasi-split inner form $G'$, we have
\begin{center}
    $d\nu_{G}(\varphi)=d\nu_{G'}(\varphi')$,  
\end{center}
if $\varphi\mapsto \varphi'$ under the embedding $\Phi_{\tp}(G)\hookrightarrow \Phi_{\tp}(G')$. 
\end{theorem}
\begin{proof}
Let $f$ be a Schwartz function on $G$. 
By \cite[Theorem 4.1]{Shelstad1979char}, there exists a Schwartz function $f'$ on $G'$ such that $f,f'$ have the stable matching integrals defined there, which we will denote by $f'\leftrightarrow f$. 
We know that
\begin{equation}\label{eROI}
    \chi_{\varphi}(f)=\chi_{\varphi'}(f')
\end{equation}
if $\varphi\mapsto \varphi'$. 

On the other hand, consider an $L$-parameters $\varphi'$ of $G'$ that does not lie in the image of the embedding $\Phi(G)\hookrightarrow \Phi(G')$.  
By \cite[Theorem 4.1.1]{ShelstadLindis}, we know that $\chi_{\varphi'}(f')=0$ if $f'\leftrightarrow f$ for some Schwartz function $f$ on $G$. 
As $f'(1)=f(1)$ when $f'\leftrightarrow f$, we have
\begin{center}
    $f(1)=\int_{\Phi_{\tp}(G)}\chi_{\varphi}(f)d\nu_G(\varphi)$ and $f'(1)=\int_{\Phi_{\tp}(G')}\chi_{\varphi'}(f')d\nu_{G'}(\varphi')$. 
\end{center}
As the second integration has its support in $\Phi_{\tp}(G)$, the formula \eqref{eROI}
implies $\nu_{G}(\varphi)=\nu_{G'}(\varphi')$ if $\varphi\mapsto \varphi'$. 
\end{proof}

We aim to apply Theorem \ref{tshelsinn} for $G'=\GL(n,\mathbb{R})$, in which case the embedding $\Phi(G)\mapsto \Phi(G')$ reduces to the local Jacquet-Langlands correspondence for real $\GL(n)$. 
We will first give this correspondence explicitly. 
The tempered dual of $\GL(n,\mathbb{R})$ and $\GL(n,\mathbb{H})$ can be described as follows. 
Observe $\GL(n,\mathbb{R})$ has square-integrable irreducible representations (modulo the center) if and only if $n=1$ or $2$. 
The set $\Pi_2(^{0}{M}_P)$ in Equation \eqref{eRPformula} is non-empty if and only if $^{0}{M}_P$ is a product of $\SL^{\pm}(2,\mathbb{R})$ and $\SL^{\pm}(1,\mathbb{R})=\{\pm 1\}$ (and the Levi factor $M={}^{0}{M}_P A$ is the products of $\GL(1,\mathbb{R})$ and $\GL(2,\mathbb{R})$). 

Hence we take $\sigma=(n_1,\cdots,n_s)$ be a partition of $n$ into $1$'s and $2$'s. 
For each $\GL(n_i,\mathbb{R})$, we take a square-integrable irreducible representation (modulo the center) which can be completely classified as follows: 
\begin{itemize}
    \item For $\GL(1,\mathbb{R})=\mathbb{R}^{\times}$, let $\chi_{t,+}(g)=|g|_{\mathbb{R}}^t$ and $\chi_{t,-}(g)=\sgn(g)\otimes |g|_{\mathbb{R}}^t$ for $t\in\mathbb{C}$.  
    \item For $\GL(2,\mathbb{R})$, let $H_k(\omega)=\omega\otimes \Ind_{\SL(2,\mathbb{R})}^{\SL^{\pm}(2,\mathbb{R})}H_k^+=\omega\otimes(H_k^+\oplus H_k^-)$ for $k\geq 1$. 
\end{itemize}
Each tempered irreducible representation $\pi$ is of the form $\Ind_{P}^{G}(\pi_1\otimes\cdots\otimes \pi_s)$ with the Levi factor $M_P\cong \prod_{1\leq i\leq s}\GL(n_i,\mathbb{R})$ and each $\pi_i$ taken from one of representations of $\GL(n_i,\mathbb{R})$ above, where $n_i=1,2$.  

Here we describe the tempered dual of $\GL(n,\mathbb{H})$. 
The maximal compact subgroup of $\GL(n,\mathbb{H})$ is $\Sp(n)$. 
It has a maximal Cartan subgroup, or equivalently $\rk \Sp(n)=\rk \GL(n,\mathbb{H})$ ($=n$) if and only if $n=1$.
Thus, $\GL(n,\mathbb{H})$ has square-integrable irreducible representations (modulo the center) if and only if $n=1$. 
They come from the following ones:
\begin{itemize}
    \item $V_k(\omega)=\omega\otimes \mathbb{C}^k$, where  $\mathbb{C}^k$ is the unique $k$-dimensional irreducible representation of $\mathbb{H}^1=\{x\in \mathbb{H}|\Nm(x)=1\}\cong\SU(2)$ and $\omega$ is a character of $\mathbb{R}_{>0}^{\times}$. 
\end{itemize}
Hence the induced representations of the form $\Ind^{G}_{\prod_{1\leq i\leq n}\mathbb{H}^{\times}}(V_{k_1}(\omega_1)\otimes \cdots \otimes V_{k_n}(\omega_n))$ give the  tempered dual of $\GL(n,\mathbb{H})$.

The local Jacquet-Langlands correspondence of $\GL(n,\mathbb{R})$ is given explicitly in \cite[\S 13]{BaduRena10} as a map on the Grothendieck groups of unitary representations: 
\begin{center}
    $\JL_{\mathbb{R}}\colon \mathbb{Z}\cdot\Pi({\GL(2n,\mathbb{R})})\to \mathbb{Z}\cdot\Pi(\GL(n,\mathbb{H}))$. 
\end{center}
The restriction of the correspondence on the tempered representations can be described as follows. 
\begin{enumerate}
    \item If each factor of the Levi subgroup $M$ is $\GL(2,\mathbb{R})$ (the partition $\sigma=(2,\cdots,2)$), then
    \begin{center}
        $\JL_{\mathbb{R}}(\Ind(H_{k_1}(\omega_1)\otimes\cdots \otimes \Ind(H_{k_n}(\omega_n)))=\Ind(\Ind(V_{k_1}(\omega_1)\otimes\cdots \otimes \Ind(V_{k_n}(\omega_n)))$
    \end{center}
    for $k_i\geq 1$ and each $\omega$ is a character of $\mathbb{R}_{>0}^{\times}$. 
    \item If there exists a factor $\GL(1,\mathbb{R})$ of the Levi subgroup $M$ (the partition $\sigma$ contains at least one $1$), then
\begin{center}
$\JL_{\mathbb{R}}(\Ind(\pi_1\otimes\cdots\otimes \pi_s))=0$
\end{center}
for any representations 
$\pi_1\otimes\cdots\otimes \pi_s$ of $M$. 
Note that $0$ stands as the zero element in the Grothendieck group but not the trivial representation. 
\end{enumerate}

As an $L$-packet of $\GL(n)$ (or its inner forms) is a single irreducible representation, Theorem \ref{tshelsinn} implies the following results. 
\begin{corollary}\label{cratioRPd}
Let $\pi,\pi'$ be nontrivial tempered irreducible representations of $\GL(2n,\mathbb{R})$ and $\GL(n,\mathbb{H})$ such that $\JL_{\mathbb{R}}(\pi)=\pi'$.
We have
\begin{center}
    $d \nu_v (\pi)=d \nu'_v(\pi')$,
\end{center}
if the Haar measures are taken to be the local Tamagawa measures above.   
\end{corollary}

\subsection{The local-global compatibility of Plancherel and Tamagawa measures}\label{scomp}

This section is devoted to the compatibility of the local Plancherel measures and the global Tamagawa measure on an adelic group (see \cite{Kott88}) with respect to the local (thus global) Jacquet-Langlands correspondence.

For a cuspidal automorphic representation, it is not measurable with respect to the Plancherel measure of an adelic group $G(\mathbb{A})$.
Thus, it is not feasible to consider the adelic Plancherel measure. 
Instead, we show that the global correspondence preserves each local Plancherel measure once the local Tamagawa measures (see \cite[\S 29]{Voight21}) are chosen as the Haar measures. 

For an $F$-group $G$, the Plancherel measure of an adelic group $G(\mathbb{A})$ is known to be the restricted product of local Plancherel measures, which is described as follows (see \cite{Y24mathZ} and \cite[\S 2]{Blkd77}). 
Let $\{(X_i,\mu_i):i\in I\}$ be a set of measure spaces. 
Let $I_0\subset I$ such that $I\backslash I_0$ is finite. 
For each $i\in I_0$, we take a subset $Y_i\subset X_i$ such that $\mu_{i}(Y_i)$ is finite and $\prod_{i\in I_0}\mu_{i}(Y_i)$ is convergent. 
For any finite subset $S\subset I$ satisfying $S\cup I_{0}=I$, we define
\begin{center}
    $(X_S,\nu_S)=(\prod_{i\in S}X_i)\times (\prod_{i\notin S}Y_i)$. 
\end{center} 
Note for $S'\subset S$, there is a natural embedding of measure spaces $X_{S'}\subset X_S$.
The {\it restricted product of $\{(X_i,\mu_i):i\in I\}$ with respect to $\{Y_i:i\in I_0\}$} is defined by
\begin{center}
$\prod'_{i\in I} (X_i,Y_i,\mu_i)=\varinjlim_S X_S=\{(x_i)_{i\in I}\in \prod X_i|x_i\in Y_i\text{~for~almost~all~}i\}$,
\end{center}
which will also be denoted simply by $(X,\nu)$. 
A subset $Z\subset X$ is measurable if and only if $Z\cap X_S$ is measurable for all finite $S$ and its measure is given by
\begin{center}
    $\mu(Z)=\sup_{S}\nu_S(Z\cap X_S)$.
\end{center} 

Now we take $I=V$, $I_0=V_f$ and $X_v=\Pi_{\tp}(G(F_v))$ together with $Y_v=\Pi_{\rm urmf}(G(F_v))$ (the set of equivalence classes of unramified irreducible representations of $G(F_v)$) for each $v\in V_f$. 
Let $\nu_v$ be the local Plancherel measure on $\Pi_{\tp}(G(F_v))$. 
Then the Plancherel measure $\nu_\mathbb{A}$ of $G(\mathbb{A})$ is the restricted product of $\{\Pi_{\tp}(G(F_v),\nu_v):v\in V\}$ with respect to $\{\Pi_{\rm urmf}(G(F_v)):i\in V_f\}$. 
Therefore, each cuspidal automorphic representation $\pi=\otimes \pi_v \in \Pi_0(G)$ is not $\nu_\mathbb{A}$-measurable as $\pi_v$ is unramified almost everywhere and thus is not measurable with respect to $\nu_v$. 
Hence, we should consider the global correspondence of each local Plancherel measure. 

We take the local Tamagawa measure $\mu_v$ as the Haar measure on the local groups $\GL(nd,F_v)$ and $\GL(n_v,D_v)$ over $F_v$ respectively, which are defined in Section \ref{ssHaarp} and Section \ref{ssHaarR}. 
For each $v$, we let $d \nu_v, d \nu'_v$ be the Plancherel density determined by $\mu_v$ with respect to the Harish-Chandra measure $d \omega_v$ on the local tempered dual $\Pi_{\tp}(\GL(n_v,D_v))$ and $\Pi_{\tp}(\GL(n_v,D_v))$ respectively. 
\begin{theorem}\label{tpreallPm}
There exists local Haar measures $\mu_v,\mu'_v$ on $G_v,G'_v$ such that 
\begin{enumerate}
    \item $\prod \mu_v,\prod \mu'_v$ are the Tamagawa measures on $G(\mathbb{A})$ and $G'(D(\mathbb{A}))$ respectively;
    \item the local Jacquet-Langlands correspondence preserves the associated Plancherel measures, i.e., $d \nu_{v}(\pi_v)=d \nu'_{v}(\pi'_v)$ if $\pi_v$ and $\pi'_v$ correspond. 
    \item the global Jacquet-Langlands correspondence of tempered cuspidal automorphic representations preserves each local Plancherel measure, i.e., $d \nu_v(\pi_v)=d \nu'_v(\pi'_v)$ for each $v$ if $\pi=\otimes_{v\in V}\pi_v$ and $\pi'=\otimes_{v\in V}\pi'_v$ correspond. 
\end{enumerate}
\end{theorem}
\begin{proof}
It is known that the product of the local Haar measures is the Tamagawa measure by \cite[\S 29.8]{Voight21}. 
The second property follows Corollary \ref{cratioRPd} and Proposition \ref{pratiopPd}. 
The last one follows the fact that $\GL(n,D_v)\cong \GL(nd,F_v)$ for $v\notin \ram(G')$ and also the the second one for $v\in \ram(G')$. 
\end{proof}

\section{An arithmetic invariant of the correspondence}\label{sJLinv}

We show the density and the dimension over the arithmetic subgroups are preserved by the Jacquet-Langlands correspondence.

\subsection{The covolume of $S$-arithmetic subgroups}

Let $G_1,G_2$ be simply connected semisimple algebraic groups over a global field $F$. 
Assume that $G_1$ is an inner form of $G_2$, that is, they become isomorphic over 
$\overline{F}$ (or a finite extension of $F$) via an inner automorphism. 
Let $G_0$ be the quasi-split inner form of $G_1,G_2$, and $\varphi_i\colon G_i\to G_0$ be an inner form defined over some finite Galois extension $K_i$ of $F$. 
We know that $\varphi_i\circ {}^{\gamma}\varphi$ is an inner automorphism of $G$ for each $\gamma\in \Gal(K_i/F)$. 

Let $\ram(G_i)$ be the set of ramified places of $G_i$, i.e., $G_i(F_v)\cong G_0(F_v)$ if and only if $v\notin \ram(G_i)$. 
Thus, for $v\notin \ram(G_1)\cup \ram(G_2)$ , we have $G_{1}(F_v)\cong G_{2}(F_v)$.

Let $\omega_0$ be an invariant differential form on $G_0$ of the maximal degree. 
Then $\omega_i=\varphi_{i}^{*}(\omega_0)$ in an invariant differential form on $G_i$ of the maximal degree for $i=1,2$. 
Let $\mu_i$ be the Haar measure on $G_i(F)$ determined by $\omega_i$. 
For each $v\in V$ (the set of places of $F$),
we follow \cite[\S 2.1]{Prasad89} to give the local Haar measures. 
For $x\in F_v^{\times}$, we let
\begin{equation*}
    |x|_v=
    \begin{cases}
      q_v^{-v(x)}, & \text{if}\ v\in V_f; \\
      |x|, & \text{if}\ F_v=\mathbb{R};\\
      |x|^2, & \text{if}\ F_v=\mathbb{C}.
    \end{cases}
  \end{equation*}
Thus, the differential form $\omega_i$, together with the normalized absolute value $|~|_v$ on $F_v$, determines a Haar measure on $G_i(F_v)$. 
We denote this Haar measure by $\mu_{i,v}$ for $i=1,2$. 
We call the Haar measures  $\mu_{1,v},\mu_{2,v}$ on $G_{1}(F_v),G_{2}(F_v)$ (resp., $\mu_1,\mu_2$ on $G_{1}(F),G_{2}(F)$) {\it compatible} if they are constructed from the same $\omega_0$ on the quasi-split form $G_0$ as above. 

Let $S$ be a finite set of places of $F$ containing all the archimedean ones and $\mathcal{O}_S$ be the ring of $S$-integers. 
We define
\begin{center}
    $\Gamma_i(S)=G_i(\mathcal{O}_S)$
\end{center}
the $S$-arithmetic subgroup of $G_i(F_S)=\prod_{v\in S}G_i(F_v)$ for $i=1,2$. 
We equipped $G_i(F_S)$ with the product measure $\mu_{i,S}=\prod_{v\in S}\mu_{i,v}$

\begin{proposition}\label{pStama}
Let $G_1,G_2$ be simply connected semisimple algebraic groups over a number field $F$ such that they are inner forms of each other. 
Suppose $\mu_{1,v},\mu_{2,v}$ are compatible Haar measures on $G_1(F_v),G_2(F_v)$. 
Let $\mu_{1,S}=\prod_{v\in S}\mu_{1,v}$ and $\mu_{2,S}=\prod_{v\in S}\mu_{2,v}$. 
We have
\begin{center}
    $\mu_{1,S}(G_{1}(F_S)/\Gamma_{1}(S))=\mu_{2,S}( G_{2}(F_S)/\Gamma_{2}(S))$
\end{center}
if $\ram(G_1)\cup \ram(G_2)\subset S$ and both $G_1(F_S),G_2(F_S)$ are non-compact.  
\end{proposition}
\begin{proof}
By the strong approxiamation property (see\cite[Theorem A]{Prasad77strongappr}), we know
\begin{equation}\label{estrongappr}
    G_{i}(F_S)\cdot \prod_{v\notin S}P_{i,v}\cdot G_{i}(F)=G_{i}(\mathbb{A}),
\end{equation}
where $\{P_{i,v}\}_{v\in V_f}$ is a collection of parahoric subgroups $P_{i,v}$ of $G_{i}(F_v)$ such that $\prod_{v\in V_{\infty}}G_{i}(F_v)\cdot \prod_{v\in V_f}P_{i,v}$ is open in $G_{i}(\mathbb{A})$. 
We take the collection $\{P_{i,v}\}_{v\in V_f}$ such that $\Gamma_{i}(S)=G_{i}(F)\cap \prod_{v\notin S_f}P_{i,v}$. 
Thus, \eqref{estrongappr} implies a fibration
\begin{center}
    $G_i(\mathbb{A})/G_{i}(F)\to G_{i}(F_S)/\Gamma_{i}(S)$,
\end{center}
whose fiber can be identified with $\prod_{v\notin S}P_{i,v}$. 
Hence we obtain
\begin{center}
    $\mu_{i,S}(G_{i}(F_S)/\Gamma_{i}(S))=D_{F}^{\frac{1}{2}\dim G_i}\tau_{F}(G_i)\cdot (\prod_{v\notin S}\mu_{i,v}(P_{i,v}))^{-1}$,
\end{center}
where $D_{F}$ denotes the discriminant of $F$ and $\tau_{F}(G_i)$ is the Tamagawa number of $G_i$. 

For $s\notin \ram(G_1)\cup \ram(G_2)$, we have both $G_{i}(F_v)$ and $G_{i}(F_v)$ isomorphic to the quasi-split group $G_{0}(F_v)$ and the last terms coincide. 
By \cite{Kott88}, it is known that $\tau_F(G_i)=\tau_F(G_0)$ for $i=1,2$. 
Hence we obtain $\mu_{1,S}( G_{1}(F_S)/\Gamma_{1}(S))=\mu_{2,S}(G_{2}(F_S)/\Gamma_{2}(S))$. 
\end{proof}

We go back to the case $G=\GL(n,F)$ and $G'$ is an inner form of $G$. 
Then we let $\overline{G}=\PGL(n,F)$ and $\overline{G'}=G'/Z(G')$ is an inner form of $\overline{G}$. 
The Haar measures $\mu_v,\mu'_v$ are taken to be the ones constructed in Section \ref{ssHaarR} and Section \ref{ssHaarp}, which are compatible (see \cite[\S 29]{Voight21}). 
Let $\mu_S,\mu'_S$ be the products of these local Haar measures over $v\in S$. 
\begin{corollary}\label{cPGLhaar=}
If $\ram(\overline{G'})\subset S$, 
$\mu_{S}(\PGL(n,F_S)/ \PGL(n,\mathcal{O}_S))=\mu'_{S}(\overline{G'}(F_S)/ \overline{G'}(\mathcal{O}_S))$
\end{corollary}
\begin{proof}
Let $R$ be a commutaive ring and consider the following short exact sequence
\begin{center}
    $1\xrightarrow[]{} \SL(n,R)\xrightarrow[]{}\GL(n,R)\xrightarrow[]{\eta}  R^{\times}\to 1$,
\end{center}
where $\eta$ is the determinant map. 
Let $\overline{\eta}$ denote the induced map $\PGL(n,R)\to R^{\times}/(R^{\times})^{n}$. 
It is known that $\ker \overline{\eta}=\PSL(n,R)$ and
\begin{center}
    $\PGL(n,R)/\PSL(n,R)\cong R^{\times}/(R^{\times})^{n}$. 
\end{center}
On the other hand, $\SL(n,R)/\mu_n(R)\cong \PSL(n,R)$, where $\mu_{n}(R)$ is the group of $n$-th roots of unity in $\mathbb{R}^{\times}$.

Let $\omega,\omega'$ be the invariant differential forms of top degree on $\overline{G}(F)=\PGL(n,F),\overline{G'}$ which determines the compatible local Haar measures $\mu_v,\mu'_{v}$. 

We first consider the split case, i.e., $\overline{G}=\PGL(n)$. 
Note that both the maps $\SL(n,R)\to \PSL(n,R)$ and $\PGL(n,R)\to \PSL(n,R)$ give coverings of finite index when $R$ is a local field. 
We may pull back $\omega$ to the larger groups. It gives compatible Haar measures on $\SL(n,F_v),\PSL(n,F_v)$, which we will also denote by $\mu_{v}$.
Thus, we obtain
\begin{equation}\label{ePGL-SL}
\begin{aligned}
    \mu_{S}(\PGL(n,F_S)/\PGL(n,\mathcal{O}_S))&=\frac{[\PGL(n,F_S):\PSL(n,F_S)]}{[\PGL(n,\mathcal{O}_S):\PSL(n,\mathcal{O}_S)]}\cdot \mu_{S}(\PSL(n,F_S)/\PSL(n,\mathcal{O}_S))\\
    &=\frac{[F_S^{\times}:(F_S^{\times})^n]}{[\mathcal{O}_S^{\times}:(\mathcal{O}_S^{\times})^n]}\cdot \frac{1}{[\mu_n(F_S):\mu_n(\mathcal{O}_S)]}\cdot \mu_{S}(\SL(n,F_S)/\SL(n,\mathcal{O}_S)). 
\end{aligned}
\end{equation}

For the non-split group $G'$ and $\overline{G'}=G'/Z(G')$, we have the reduced norm $\nrd$ which gives the short exact sequence
\begin{center}
    $1\xrightarrow[]{} G'^{(1)}(R)\xrightarrow[]{}G'(R)\xrightarrow[]{\nrd}  R^{\times}\to 1$,
\end{center}
where $G'^{(1)}=\{g\in G'|\nrd(g)=1\}$. 
Let $\overline{\nrd}$ denote the induced map on $(G'/Z(G'))(R)\to R^{\times}/(R^{\times})^{n}$.
We denote $(G'^{(1)}/Z(G'^{(1)}))$ by $\overline{G'^{(1)}}$. 
We know that $\ker\nrd=\overline{G'^{(1)}}(R)$, which is an inner form of $\PSL(n,R)$ when $R$ is a local field.

On the other hand, we have $G'^{(1)}(R)/\mu_{n}(R)\cong \overline{G'^{(1)}}(R)$. 
The same arguments as in Equation \eqref{ePGL-SL} give
\begin{equation}\label{ePG'G'1}
\begin{aligned}
    \mu'_{S}(\overline{G'}(n,F_S)/\overline{G'}(n,\mathcal{O}_S))&=\frac{[\overline{G'}(n,F_S):\overline{G'^{(1)}}(n,F_S)]}{[\overline{G'}(n,\mathcal{O}_S):\overline{G'^{(1)}}(n,\mathcal{O}_S)]}\cdot \mu_{S}(\overline{G'^{(1)}}(n,F_S)/\overline{G'^{(1)}}(n,\mathcal{O}_S))\\&=\frac{[F_S^{\times}:(F_S^{\times})^n]}{[\mathcal{O}_S^{\times}:(\mathcal{O}_S^{\times})^n]}\cdot \frac{1}{[\mu_n(F_S):\mu_n(\mathcal{O}_S)]}\cdot \mu'_{S}(G'^{(1)}(n,F_S)/G'^{(1)}(n,\mathcal{O}_S)). 
\end{aligned}
\end{equation}
By Theorem \ref{pStama}, the comparison of Equations \eqref{ePGL-SL} and \eqref{ePG'G'1} gives the desired formula. 
\end{proof}

\subsection{The correspondence of densities over $S$-arithmetic groups}

Let $G=\GL(n)$ and $G'=D^{\times}$ (the multiplicative group of a central division algebra over $F$ of dimension $F$). 
Let $S$ be a finite set of places which contains $V_{\infty}$. 
We define two $S$-arithmetic groups
\begin{center}
    $\Gamma_{S}=\overline{G}(\mathcal{O}_S)$ and $\Gamma'_{S}=\overline{G'}(\mathcal{O}_S)$,
\end{center}
which are usually called the {\it principal $S$-arithmetic subgroups} of the semisimple quotient groups $\overline{G}=\PGL(n)$ and $\overline{G'}=G'/Z(G')$. 

When passing from an ordinary representation $\pi$ of a group to the projective representation of the quotient group, we let $\sigma(\pi)$ denote the $2$-cocycle associated with $\pi$.  
Note that for $S$-arithmetic groups, their second cohomology groups are not trivial in general (see \cite{AdemNaf98}). 
\begin{theorem}\label{ttpvndim}
    Suppose $\pi,\pi'$ are cuspidal automorphic representations of $G(\mathbb{A})$ and $G'(\mathbb{A})$ that are matched by the Jacquet-Langlands correspondence.   
We have
\begin{center}
$d_{\Gamma_S,\sigma(\pi_S)}\pi_S=d_{\Gamma'_S,\sigma(\pi'_S)}\pi'_S$
\end{center}
for any finite set $S$ of places such that  $\ram(G')\subset S$, and both $G(F_S),G'(F_S)$ are not compact modulo the center.  
\end{theorem}
\begin{proof}
By Theorem \ref{tpreallPm}, the local and global Jacquet-Langlands correspondence preserves each local Plancherel measure $\nu_v$ (resp. $\nu'_v$) of $G(F_v)$ (resp. $G'(F_v)$). 
Thus, we know $d\nu_v(\pi_v)=d\nu'_v(\JL_{v}(\pi_v))$ for all place $v$.

Note that  $\pi_S,\pi'_S$ give projective representations of the projective groups $\overline{G}_S$ and $\overline{G'}_S$.  
By Theorem \ref{tcenextPfml}, there is a decomposition of $\nu_v$ into the Plancherel measures of the center $Z(G_v)$ and of the projective representations of the quotient group $\overline{G_v}=G_v/Z(G_v)$: 
\begin{center}   $d\nu_v(\pi(\tau,\chi))=d\nu_{\overline{G_v},\overline{\sigma_{\chi}}}(\tau)d \nu_{Z(G_v)}(\overline{\chi})$,
\end{center}
where $\chi\in\widehat{Z(G_v)}$ and $\tau\in \Pi(\overline{G_v},\sigma_{\chi})$. 
Similarly, for the decomposition of $\nu'_v$ on $G'_v$, we have
\begin{center} $d\nu'_v(\pi(\tau',\chi'))=d\nu_{\overline{G'_v},\overline{\sigma_{\chi'}}}(\tau')d \nu_{Z(G'_v)}(\overline{\chi'})$
\end{center}
where $\chi'\in\widehat{Z(G'_v)}$ and $\tau'\in \Pi(\overline{G'_v},\sigma_{\chi'})$. 
 
Since local Jacquet-Langlands correspondence preserves central characters,  
the $\sigma_{\chi}$-projective Plancherel measure $\nu_{\overline{G_v},\overline{\sigma_{\chi}}}(\tau)$  is also preserved by the local Jacquet-Langlands correspondence. 
Now we apply Theorem \ref{tdimmeas} for groups $\overline{G_v}$ and $\overline{G'_v}$. 
It then follows Corollary \ref{cPGLhaar=} for the factors of covolumes.  
\end{proof}

\begin{remark}
\begin{enumerate}
    \item The last condition on the non-compactness is equivalent to saying that there exists $v\in S$ such that both $G_v,G'_v$ are isotropic (modulo the center), which is necessary for the strong approximation property. 
    It is also equivalent to the condition that $\Gamma_S,\Gamma'_S$ are both infinite. 
    \item While the theorem applies to tempered representations, we can take the density of non-tempered representations to be $0$. 
    \item The Ramanujan-Petersson conjecture states that each cuspidal automorphic representation of $\GL(n)$ is tempered.
This was proved completely for function field cases by Drinfeld \cite{Drin88} for $n=2$ and by L. Lafforgue \cite{LafL02} in general.  
It remains open for number fields except for a few special families of cuspidal representations (see \cite{Lww2020}).
\end{enumerate}

\end{remark}

\begin{corollary}\label{csqvndim}
Suppose $\pi,\pi'$ are cuspidal automorphic representations of $G(\mathbb{A})$ and $G'(\mathbb{A})$ that are matched by the Jacquet-Langlands correspondence. 
Assuming both $\pi_S,\pi'_S$ are square-integrable representations, we have
\begin{center}
    $\dim_{\mathcal{L}(\Gamma_S,\sigma(\pi_S))}\pi_S=\dim_{\mathcal{L}(\Gamma'_S,\sigma(\pi'_S))}\pi'_S$. 
\end{center}
if $\ram(G')\subset S$ and both $G(F_S),G'(F_S)$ are not compact modulo their centers.
\end{corollary}

For $n> 2$, the last condition on non-compactness is always satisfied. 
If $F$ is a number field and $V_{\infty}\subset S$, a necessary condition for $\pi_S$ being a square-integrable irreducible representation is $F$ totally real and $G=\GL(2)$. 
This comes from the fact that $\GL(n,\mathbb{C})$ has square-integrable irreducible representations if and only if $n=1$ while $\GL(n,\mathbb{R})$ has square-integrable irreducible representations if and only if $n=1,2$.

{
{
\setstretch{0.2}
\bibliographystyle{abbrv}
\typeout{}
\bibliography{MyLibrary} 
}
}

\textit{E-mail address}: \href{mailto:junyang@fas.harvard.edu}{junyang@fas.harvard.edu}

\end{document}